\documentclass[12pt]{article}
\usepackage[hidelinks]{hyperref} 
\usepackage{amsthm}
\usepackage{amsmath}
\usepackage{enumitem}
\usepackage{amssymb}
\usepackage{multicol}
\usepackage{geometry}
\usepackage{mathtools}

\usepackage{thmtools}
\usepackage{thm-restate}

\usepackage{MnSymbol}
\usepackage[all]{xy}

\title{Characterizing relative decidability\\in terms of model completeness}
\author{Matthew Harrison-Trainor and Liam Tan\thanks{Harrison-Trainor was partially supported by a Sloan Research Fellowship and by the National Science Foundation under Grant DMS-2452105 and DMS-2419591/DMS-2153823. Liam Tan was supported as an REU student by the National Science Foundation under grant DMS-2419591.}}

\newcommand{\cB}{\mathcal{B}}
\newcommand{\cD}{\mathcal{D}}
\newcommand{\cC}{\mathcal{C}}

\newcommand{\cL}{\mathcal{L}}
\newcommand{\cM}{\mathcal{M}}
\newcommand{\cA}{\mathcal{A}}

\newcommand{\cN}{\mathcal{N}}

\DeclareMathOperator\res{\upharpoonright}

\DeclareMathOperator{\Th}{Th}

\newcommand{\barx}{\bar{x}}
\newcommand{\bary}{\bar{y}}

\newcommand{\bara}{\bar{a}}
\newcommand{\barc}{\bar{c}}
\newcommand{\barb}{\bar{b}}

\DeclareMathOperator{\Diag}{Diag}
\newcommand{\set}[1]{\{ #1 \}}

\newcommand{\vphi}{\varphi}
\newcommand{\abs}[1]{\left| #1 \right|}
\newcommand{\lra}{\leftrightarrow}

\newcommand{\mc}[1]{\mathcal{#1}}

\DeclareMathOperator{\flim}{Flim}
\DeclareMathOperator{\age}{Age}

\newtheorem{theorem}{Theorem}[section]
\newtheorem{lemma}[theorem]{Lemma}

\newcounter{claims}[theorem]
\newtheorem{claim}[claims]{Claim}

\newtheorem*{subclaim}{Subclaim}

\theoremstyle{definition}
\newtheorem{definition}[theorem]{Definition}

\newtheorem{question}[theorem]{Question}

\newcounter{cases}[theorem]
\newtheorem{case}[cases]{Case}

\theoremstyle{remark}

\newcommand{\treeT}{\mathfrak{T}}

\begin{document}
	
	\maketitle
	
	\begin{abstract}
		A theory $T$ is said to be relatively decidable if for every model of $T$, one can compute the elementary diagram of that model from its atomic diagram together with $T$. We verify a conjecture of Chubb, Miller, and Solomon by showing that for complete theories $T$, $T$ is relatively decidable if and only if $T$ has a conservative model complete extension of the form $T \cup \{\varphi(\bar{c})\}$ where $T \models \exists \bar{x} \; \varphi(\bar{x})$. We also show that no such characterization works for incomplete theories.
	\end{abstract}
	
	\section{Introduction}
	
	In applied model theory, especially in the style of Robinson, one of the main tasks when studying a particular class of structures is to understand the definable sets. The strongest way to do this is to prove a quantifier elimination result, which says that every definable set can be defined by a formula without quantifiers. Chevalley's theorem on constructible sets is exactly quantifier elimination for algebraically closed fields, while work of Seidenberg, Robinson, and Blum established quantifier elimination for differentially closed fields in characteristic zero (see \cite{diff}). Many natural structures do not have quantifier elimination in their natural languages, but instead satisfy a weaker variant of quantifier elimination known as model completeness: Every definable set can be defined by both an existential formula and a universal formula. A classical example due to Tarski \cite{tarski} is the field of real numbers, which admits quantifier elimination in the language of ordered fields but is only model complete in the algebraic language without the ordering. Other examples of model completeness without quantifier elimination include the field $\mathbb{Q}_p$ of $p$-adic numbers \cite{padic}, the exponential field $(\mathbb{R},\exp)$ of real numbers \cite{exp}, and algebraically closed fields with a generic automorphism \cite{difference}.
	
	One consequence of model completeness is that it yields an algorithm reducing the problem of computing membership in definable sets to membership in quantifier-free definable sets. Such an algorithm might be thought of as the weakest possible variant of quantifier elimination. Our main result is that for complete theories the existence of such an algorithm implies model completeness after naming finitely many constants. This perhaps explains the ubiquity of model completeness.
	
	More formally, let $\mc{L}$ be a computable first-order language and $T$ a recursively axiomatizable first-order theory. If $T$ is model complete, then every formula $\varphi(\bar{x})$ is equivalent modulo $T$ to both an existential and a universal formula. For every $\mc{A} \models T$ (with domain $\mathbb{N}$, so that it makes sense to perform computations on $\mc{A}$) we can use the atomic diagram of $\mc{A}$ to compute the elementary diagram of $\mc{A}$. Indeed, given a formula $\varphi(\bar{x})$ and $\bar{a} \in \mc{A}$, we can find quantifier-free formulas $\psi_{\mathsf{T}}(\bar{x},\bar{y})$ and $\psi_{\mathsf{F}}(\bar{x},\bar{z})$ such that
	\[ T \models \varphi(\bar{x}) \longleftrightarrow \exists \bar{y} \; \psi_{\mathsf{T}}(\bar{x},\bar{y}) \quad \text{ and } T \models \neg \varphi(\bar{x}) \longleftrightarrow \exists \bar{z} \;  \psi_{\mathsf{F}}(\bar{x},\bar{z}).\]
	Since the theory is c.e., we can search for these formulas computably. Then, using the atomic diagram of $\mc{A}$, we can simultaneously search for either $\bar{b} \in \mc{A}$ such that $\mc{A} \models \psi_{\mathsf{T}}(\bar{a},\bar{b})$ (in which case we have determined that $\mc{A} \models \varphi(\bar{a})$) or for $\bar{c} \in \mc{A}$ such that $\mc{A} \models \psi_{\mathsf{F}}(\bar{a},\bar{c})$ (in which case we have determined that $\mc{A} \models \neg \varphi(\bar{a})$).
	
	Chubb, Miller, and Solomon \cite{ChubbMillerSolomon} named this property of $T$ relative decidability. While they defined this only for recursively axiomatizable theories, we remove this assumption by modifying the definition to allow access to the theory $T$ in the computations.
	
	\begin{definition}
		A theory $T$ is relatively decidable if for every countable $\mc{A} \models T$ with domain $\mathbb{N}$ the theory $T$ together with the atomic diagram of $\mc{A}$ compute the elementary diagram of $\mc{A}$.
	\end{definition}
	
	\noindent If $T$ is model complete, then not only is $T$ relatively decidable, but it is uniformly relatively decidable: There is a single algorithm (or Turing functional) that can always be used to compute the elementary diagram of a model from the atomic diagram. Chubb, Miller, and Solomon \cite{ChubbMillerSolomon} showed that these are the only uniformly relatively decidable theories, i.e., that $T$ is uniformly relatively decidable if and only if it is model complete. This result holds for both complete and incomplete theories.
	
	Chubb, Miller, and Solomon provided several examples of natural theories that are relatively decidable but not uniformly relatively decidable, including the theory of $(\mathbb{N},S)$ where $S$ is the successor relation. The non-uniformity comes from the fact that in each model of $\Th(\mathbb{N},S)$ we must identify the unique element that has no predecessor. In particular,
	\[ \Th(\mathbb{N},S) \models \exists x \forall y \; S(y) \neq x\]
	and in the language $\{S,0\}$ with an additional constant, $\Th(\mathbb{N},S) \cup \{\forall y \; S(y) \neq 0 \}$ is model complete. Chubb, Miller, and Solomon \cite{ChubbMillerSolomon} conjectured that one can characterize relative decidability in this form.
	
	Our main result is a confirmation of this conjecture when the theory $T$ is complete.\footnote{Several years ago the first author circulated an incorrect refutation of this conjecture. He apologizes.}
	
	\begin{restatable}{theorem}{maintheorem}\label{thm:main}
		Let $T$ be a complete theory. Then $T$ is relatively decidable if and only if there is a formula $\varphi(\bar{x})$ such that
		\begin{enumerate}
			\item $T \models \exists \bar{x} \; \varphi(\bar{x})$, and
			\item the $\mc{L} \cup \{\bar{c}\}$-theory $T \cup \{\varphi(\bar{c})\}$ is model complete.
		\end{enumerate}
	\end{restatable}
	
	\noindent The difficult direction of the proof is to show that if $T$ is relatively decidable then such a formula $\varphi(\bar{x})$ exists; the other implication is easy to see and is true even if the theory is not complete. The general outline is as follows. After applying standard computability-theoretic forcing techniques, for each model $\mc{A}$ and formula $\varphi(\bar{x})$ we obtain an infinitary $\Sigma_1$ formula $\psi(\bar{x})$, a countable disjunction of existential formulas, which is equivalent to $\varphi(\bar{x})$ in $\mc{A}$. The difficulty is that in different models $\mc{A}$, we may get different $\Sigma_1$ formulas. The heart of the proof is an elementary but intricate purely model-theoretic argument to resolve this by passing between large (saturated) and small (type-omitting) models of $T$. The argument, appearing in Section \ref{sec:complete}, is only four pages and is very much worth reading.
	
	Our result and the result of Chubb, Miller, and Solomon in the uniform case fit into a central theme in computable structure theory that robust computational properties should arise from definability. See, e.g., the standard texts \cite{AK00,MontalbanBook,MontalbanBook2} on computable structure theory where this theme plays a central role. Usually this is at the level of a single isomorphism type, e.g., Knight \cite{Knight86} showed that if we can enumerate a set $X$ from the atomic diagram of every copy of $\mc{A}$ then $X$ is enumeration reducible to the existential type of a tuple $\bar{a}$ in $\mc{A}$. The definability is also usually in the infinitary logic $\mc{L}_{\omega_1 \omega}$. Because we are working with a first-order theory, compactness and omitting types allow us to reduce this definability to a model-theoretic property of finitary logic. The relationship between relative decidability and uniform relative decidability also follows the usual pattern when working with a single isomorphism type, which is that in the uniform case one does not have to name a tuple from the structure, but in the non-uniform case one does. Our result agrees with this pattern on a surface level but, while usually the uniform and non-uniform cases are essentially the same, here the non-uniform case is significantly more challenging and requires new tools, e.g., saturation and omitting types, while the uniform case required only compactness.
	
	In the proof of the above theorem it is vital that $T$ be a complete theory. In the proof, this assumption is used in the following form: Given a sentence $\varphi$, if we can show that some particular $\mc{A} \models T$ satisfies $\varphi$, then every model of $T$ satisfies $\varphi$. Of course, it is natural to ask whether relative decidability can be characterized for incomplete theories as well. We show that there is no such characterization. Consider the naive definition of relative decidability; it requires a universal quantifier over all models of $T$. In terms of descriptive set theory, this makes the set of relatively decidable theories a $\boldsymbol{\Pi}^1_1$ set. On the other hand, a characterization such as that in Theorem \ref{thm:main} uses only quantifiers over finitary formulas and proofs, and so one consequence of that characterization is that the set of \textit{complete} relatively decidable theories is Borel.
	
	We show that the set of (not necessarily complete) relatively decidable theories is $\boldsymbol{\Pi}^1_1$-complete, and in particular it is not Borel.
	
	\begin{restatable}{theorem}{secondtheorem}\label{thm:main2}
		There is a language $\mathcal{L}$ such that the set
		\[ \{ \text{$T$ a $\mathcal{L}$-theory} : \text{$T$ is relatively decidable}\}\]
		is $\boldsymbol{\Pi}^1_1$-Wadge-complete.
	\end{restatable}
	
	\noindent In particular, we show that given a tree $\treeT \subseteq \mathbb{N}^{<\mathbb{N}}$, we can construct (in a relatively computable way) a theory $T_{\treeT}$ such that
	\[ \text{$\treeT$ is well-founded} \quad \Longleftrightarrow \quad \text{$T_{\treeT}$ is relatively decidable}.\]
	This implies that there is no characterization of relative decidability for incomplete theories which is simpler in form than the naive definition. In particular, we cannot prove Theorem \ref{thm:main}, or any similar theorem, for arbitrary theories. The argument uses the tool of Marker extensions of theories, and we must develop the necessary definitions and lemmas, and particularly Lemma \ref{lem:qe}, in Section \ref{sec:three}. The proof of Theorem \ref{thm:main2} follows in Section \ref{sec:four}. The definition of the theory and the argument that if $\treeT$ is well-founded then $T_{\treeT}$ is relatively decidable are both reasonably short. The argument that if $\treeT$ is ill-founded then $T_{\treeT}$ is not relatively decidable is quite technical.
	
	This situation should remind the reader of omitting types. For complete theories, we have a simple characterization of when a type can be omitted: it must be non-isolated. For incomplete theories, this is sufficient but not necessary, and the sufficient and necessary conditions obtained by Casanovas and Farr\'e \cite{CasanovasFarre} involve the well-foundedness of trees. Indeed, it is not hard to show after reading that paper that the set of omittable types is $\boldsymbol{\Sigma}^1_1$-complete.
	
	Finally, we would like to highlight the following open problem of Goncharov:
	
	\begin{question}[Goncharov]
		Characterize the decidable theories $T$ such that every computable model of $T$ is decidable.
	\end{question}
	
	\noindent The corresponding index set is $\Sigma_{\omega+2}$; we conjecture that it is $\Sigma_{\omega+2}$ $m$-complete. However, the methods from this paper are not applicable to this problem, as (by index set complexity calculations) to prove that the index set is $\Sigma_{\omega+2}$ $m$-complete one would have to build decidable models $\mc{A}$ of $T$ such that for no $\bar{a} \in \mc{A}$ is $\Th(\mc{A},\bar{a})$ model complete.
	
	\section{Complete theories}\label{sec:complete}
	
	We will require a variant of the standard omitting types theorem where, for a complete theory, we try to omit countably many partial types even if some of them may be isolated. Of course an isolated type cannot be omitted, but we can force all of its realizations to be of a certain kind. Recall that a partial type $p(\bar{x})$ is isolated if there is a formula $\varphi(\bar{x})$ such that $T \models \exists \bar{x} \; \varphi(\bar{x})$ and $T \models \forall \bar{x} \; (\varphi(\bar{x}) \longrightarrow p(\bar{x}))$. In this case we say that $\varphi(\bar{x})$ isolates $p(\bar{x})$. The formula $\varphi(\bar{x})$ does not necessarily have to be part of the type $p(\bar{x})$. Thus a realization of $p(\bar{x})$ might or might not satisfy $\varphi(\bar{x})$.
	
	\begin{definition}
		Let $T$ be a complete theory, $\mc{A} \models T$, and let $p(\bar{x})$ be a partial type. We say that a realization $\bar{a}$ of $p(\bar{x})$ in $\mc{A}$ is \textit{isolated} if there is a formula $\varphi(\bar{x})$ that isolates $p(\bar{x})$ and such that $\mc{A} \models \varphi(\bar{a})$.
	\end{definition}
	
	\begin{theorem}[Omitting types]
		Let $T$ be a complete theory, and let $(p_i)_{i \in \omega}$ be a countable list of partial types. There is a countable model $\mc{A} \models T$ such that any realization of any $p_i$ in $\mc{A}$ is isolated.
	\end{theorem}
	\begin{proof}
		We can construct $\mc{A}$ by the same construction as the standard type omitting argument except that we try to omit all of the types $p_i$ whether or not they are isolated. The result is that all realizations of any $p_i$ in $\mc{A}$ are isolated.
	\end{proof}
	
	The following lemma isolates the computability-theoretic content of the argument. It is more or less a standard application of computability-theoretic forcing with countable structures, which is essentially the same as the method of Vaught transforms.
	
	\begin{lemma}\label{lem:forcing-characterization}
		Let $T$ be a relatively decidable theory. Then for all $\mc{A} \models T$ there is $\bar{a} \in \mc{A}$ and a uniformly $T$-computable sequence of $T$-computable infinitary $\Sigma_1$ formulas $\Theta_{\varphi}(\bar{x},\bar{y})$ such that for all $\varphi$ we have
		\[ \mc{A} \models \varphi(\bar{x}) \leftrightarrow \Theta_{\varphi}(\bar{x},\bar{a}).\]
	\end{lemma}
	
	This condition is both sufficient and necessary, but it is a $\boldsymbol{\Pi}^1_1$ condition on $T$ and will be superseded by the better characterization obtained in Theorem \ref{thm:main}. Thus we prove only one direction.
	
	\begin{proof}
		Suppose that $T$ is relatively decidable. Fix $\mc{A} \models T$. Following Montalb\'an \cite{MontalbanBook2} we can consider the elementary diagram of $\mc{A}$ as a relation on $\omega \times \mc{A}^{< \omega}$,
		\[ \Diag_{el}(\mc{A}) = \{ (\ulcorner \varphi \urcorner,\bar{a}) : \mc{A} \models \varphi(\bar{a}) \}.\]
		where $\ulcorner \varphi \urcorner$ is the G\"odel number of $\varphi$. Since $T$ is relatively decidable, the atomic diagram of $\mc{A}$ computes the elementary diagram of $\mc{A}$, and hence the elementary diagram of $\mc{A}$ is relatively intrinsically computable. By the Ash-Knight-Manasse-Slaman-Chisholm theorem \cite{AshKnightManasseslaman89, Chisholm90}, in the version presented as Theorem {II.16} in Montalb\'an's book \cite{MontalbanBook}, there is $\bar{a} \in \mc{A}$ and a uniformly computable sequence of computable $\Sigma_1$ formulas $\Theta_{\varphi}(\bar{x},\bar{y})$ such that
		\[ \mc{A} \models \varphi(\bar{x}) \longleftrightarrow \Theta_\varphi(\bar{x},\bar{a}).\qedhere\]
	\end{proof}
	
	\maintheorem*
	
	\begin{proof}
		For the easy direction, suppose that there is a formula $\varphi(\bar{x})$ such that $T \models \exists \bar{x} \; \varphi(\bar{x})$ and such that the $\mc{L} \cup \{\bar{c}\}$-theory $T \cup \{\varphi(\bar{c})\}$ is model complete. Let $\mc{A} \models T$. Then there is some $\bar{a} \in \mc{A}$ such that $\mc{A} \models \varphi(\bar{a})$, and so $(\mc{A},\bar{a}) \models T \cup \{\varphi(\bar{c})\}$. Since this theory is model complete, $T$ together with the atomic diagram of $(\mc{A},\bar{a})$ computes the elementary diagram of this structure. Since $\bar{a}$ is a finite tuple of elements, the atomic and elementary diagrams of $(\mc{A},\bar{a})$ are Turing equivalent to the atomic and elementary diagrams of $\mc{A}$ respectively.
		
		Now suppose that $T$ is relatively decidable. Then by Lemma \ref{lem:forcing-characterization} for each $\mc{A} \models T$ there is $\bar{a} \in \mc{A}$ and a uniformly computable sequence of computable $\Sigma_1$ formulas $\Theta_{\varphi}(\bar{x},\bar{y}) = \bigdoublevee_{i = 1}^{\infty} \theta_{\varphi,i}(\bar{x},\bar{y})$, with each $\theta_{\varphi,i}(\bar{x},\bar{y})$ a finitary existential formula, such that for all $\varphi$ we have
		\[ \mc{A} \models \forall \bar{x} \; \left[ \; \varphi(\bar{x}) \; \longleftrightarrow \;\bigdoublevee_{i = 1}^{\infty} \theta_{\varphi,i}(\bar{x},\bar{a}) \; \right].\]
		Recall that these formulas came from a computability-theoretic forcing argument and that they are specific to the model $\mc{A}$. During the argument we will consider such formulas coming from two different models $\mc{B}$ and $\mc{C}$, and so we will use superscripts (e.g., $\theta^{\mc{B}}_{\varphi,i}(\bar{x},\bar{y})$) on these formulas to denote the associated model. The rest of the argument will be almost entirely model-theoretic (and though we use, e.g., recursive saturation, really all we need is countability rather than computability).
		
		Now let $\mc{B}$ be a recursively saturated model of $T$. As in Lemma \ref{lem:forcing-characterization} let $\bar{b} \in \mc{B}$ and $\Theta^{\mc{B}}_{\varphi}(\bar{x},\bar{y}) = \bigdoublevee_{i = 1}^{\infty} \theta^{\mc{B}}_{\varphi,i}(\bar{x},\bar{y})$ be such that for all $\varphi$,
		\begin{equation} \mc{B} \models \forall \bar{x} \;\left[ \; \varphi(\bar{x}) \;\longleftrightarrow \;\bigdoublevee_{i = 1}^{\infty} \theta^{\mc{B}}_{\varphi,i}(\bar{x},\bar{b})\; \right].\label{eq:B}\end{equation}
		We claim that for each $\varphi$ there is $n_\varphi$ such that
		\[ \mc{B} \models \forall \bar{x} \;\left[ \; \varphi(\bar{x}) \;\longleftrightarrow \; \bigvee_{i = 1}^{n_\varphi} \theta^{\mc{B}}_{\varphi,i}(\bar{x},\bar{b}) \; \right].\]
		If not, then for some $\varphi$ there is a sequence of tuples $\bar{c}_j$ with
		\[ \mc{B} \models \varphi(\bar{c}_j) \;\wedge \; \bigwedge_{i = 1}^{j} \neg \theta^{\mc{B}}_{\varphi,i}(\bar{c}_j,\bar{b}).\]
		Consider the partial type $q(\bar{x}) = \{\varphi(\bar{x}),\neg \theta^{\mc{B}}_{\varphi,i}(\bar{x},\bar{b}) : i \in \omega\}$. This type is computable and finitely realizable in $\mc{B}$, and therefore as $\mc{B}$ is recursively saturated the type must be realized in $\mc{B}$. But by \eqref{eq:B} no such type can be realized in $\mc{B}$. Thus we conclude that for each $\varphi$ there is $n_\varphi$ such that
		\[ \mc{B} \models \forall \bar{x} \;\left[ \; \varphi(\bar{x}) \;\longleftrightarrow \; \bigvee_{i = 1}^{n_\varphi} \theta^{\mc{B}}_{\varphi,i}(\bar{x},\bar{b}) \; \right].\]
		In particular, for each $\varphi$,
		\[ \mc{B} \models \exists \bar{y} \; \forall \bar{x} \; \left[ \; \varphi(\bar{x}) \; \longleftrightarrow \; \bigvee_{i = 1}^{n_\varphi} \theta^{\mc{B}}_{\varphi,i}(\bar{x},\bar{y}) \; \right].\]
		Thus, as $T$ is complete, for each $\varphi$ there is $n_\varphi$ such that
		\begin{equation}
			T \models \exists \bar{y} \; \forall \bar{x} \; \left[ \; \varphi(\bar{x}) \; \longleftrightarrow \; \bigvee_{i = 1}^{n_\varphi} \theta^{\mc{B}}_{\varphi,i}(\bar{x},\bar{y}) \; \right].\label{factfromB}
		\end{equation}
		
		Let $\Theta_{e,\varphi}(\bar{x},\bar{y}) = \bigdoublevee_{i = 1}^{\infty} \theta_{e,\varphi,i}(\bar{x},\bar{y})$ be a listing (indexed by $e$) of the computable sequences $\{ \Theta_{e,\varphi}(\bar{x},\bar{y})\}_{\varphi}$ (indexed by $\varphi$) of computable $\Sigma_1$ formulas. For each $e$, consider the partial type
		\[p_e(\bar{y}) = \left\{ \forall \bar{x} \;\left( \;\theta_{e,\varphi,i}(\bar{x},\bar{y}) \; \longrightarrow \; \varphi(\bar{x}) \; \right) \;:\; \varphi,i \right\}.\]
		By the omitting types theorem there is a model $\mc{C} \models T$ such that any realization of any $p_e$ is isolated. As in Lemma \ref{lem:forcing-characterization} let $\bar{c} \in \mc{C}$ and $\Theta^{\mc{C}}_{\varphi}(\bar{x},\bar{y}) = \bigdoublevee_{i = 1}^{\infty} \theta^{\mc{C}}_{\varphi,i}(\bar{x},\bar{y})$ be such that
		\[ \mc{C} \models \forall \bar{x} \; \left[ \; \varphi(\bar{x}) \; \longleftrightarrow \; \bigdoublevee_{i = 1}^{\infty} \theta^{\mc{C}}_{\varphi,i}(\bar{x},\bar{c}) \;\right].\]
		Let $e$ be the index of the sequence $\Theta^{\mc{C}}_{\varphi}(\bar{x},\bar{y})$. Thus $\bar{c}$ realizes $p_{e}(\bar{y})$. This realization is isolated, say by a formula $\psi(\bar{y})$. In particular, for each formula $\varphi$ and each $i$,
		\begin{equation}
			T \models \forall \bar{y} \; \left[ \; \psi(\bar{y}) \; \longrightarrow \; \forall \bar{x} \; \left( \; \theta^{\mc{C}}_{\varphi,i}(\bar{x},\bar{y}) \; 
			\longrightarrow \; \varphi(\bar{x}) \;\right)\;\right].\label{factfromC}
		\end{equation}
		Now by \eqref{factfromB}, for each $\varphi$,
		\[ \mc{C} \models \exists \bar{y} \; \forall \bar{x} \; \left[ \; \varphi(\bar{x}) \; \longleftrightarrow \; \bigvee_{i = 1}^{n_\varphi} \theta^{\mc{B}}_{\varphi,i}(\bar{x},\bar{y}) \;\right].\]
		Let $\bar{d}_\varphi$ be this $\bar{y}$, and let
		\[ \chi_\varphi(\bar{y}) := \forall \bar{x} \; \left[ \; \varphi(\bar{x}) \; \longleftrightarrow  \; \bigvee_{i = 1}^{n_\varphi} \theta^{\mc{B}}_{\varphi,i}(\bar{x},\bar{y}) \; \right].\]
		Then since $\mc{C} \models \chi_\varphi(\bar{d}_\varphi)$, for some $i_\varphi$ we have
		\[ \mc{C} \models \theta^{\mc{C}}_{\chi_\varphi,i_\varphi}(\bar{d}_\varphi,\bar{c}).\]
		Thus, for each $\varphi$,
		\begin{equation}\label{Tfactone}
			T \models \exists \bar{x} \; \exists \bar{y} \; \left[ \; \psi(\bar{y}) \wedge \theta^{\mc{C}}_{\chi_\varphi,i_\varphi}(\bar{x},\bar{y})\;\right]
		\end{equation}
		and also, by \eqref{factfromC},
		\begin{equation}\label{Tfacttwo}
			T \models \forall \bar{x} \; \forall \bar{y} \; \left[ \; \left( \; \psi(\bar{y}) \wedge \theta^{\mc{C}}_{\chi_\varphi,i_\varphi}(\bar{x},\bar{y}) \right) \; \longrightarrow \; \chi_\varphi(\bar{x}) \; \right].
		\end{equation}
		Recalling the definition of $\chi_\varphi(\bar{x})$, this says that if we find $\bar{x},\bar{y}$ with $\psi(\bar{y}) \wedge \theta^{\mc{C}}_{\chi_\varphi,i_\varphi}(\bar{x},\bar{y})$ then, over the parameters $\bar{x}$, $\varphi$ has a finitary existential definition $\varphi(\bar{z}) \leftrightarrow \bigvee_{i = 1}^{n_\varphi} \theta^{\mc{B}}_{\varphi,i}(\bar{z},\bar{x})$.
		
		For simplicity in what follows, let
		\[ \nu_\varphi(\bar{x},\bar{y}) := \theta^{\mc{C}}_{\chi_\varphi,i_\varphi}(\bar{x},\bar{y}) \]
		and let
		\[ \eta_\varphi(\bar{z},\bar{x}) := \bigvee_{i = 1}^{n_\varphi} \theta^{\mc{B}}_{\varphi,i}(\bar{z},\bar{x}).\]
		The important facts are that these are existential formulas and that, rewriting \eqref{Tfactone} and \eqref{Tfacttwo} in terms of this new notation,
		\begin{equation}\label{Tfactonenew}
			T \models \exists \bar{x} \; \exists \bar{y} \; \left[ \; \psi(\bar{y}) \; \wedge \; \nu_\varphi(\bar{x},\bar{y}) \;\right]
		\end{equation}
		and
		\begin{equation}\label{Tfacttwonew}
			T \models \forall \bar{x} \; \forall \bar{y} \; \left[ \; \left( \; \psi(\bar{y}) \wedge \nu_\varphi(\bar{x},\bar{y}) \; \right) \; \longrightarrow \; \forall \bar{z} \; \left( \; \varphi(\bar{z}) \; \longleftrightarrow  \; \eta_\varphi(\bar{z},\bar{x}) \;\right)\; \right].
		\end{equation}
		Otherwise, we can forget the specific forms that these formulas take. Note that this gives us for each $\varphi$ a formula
		\[ \exists \bar{x} \; \exists \bar{y}  \; \left[ \; \psi(\bar{y}) \; \wedge \; \nu_\varphi(\bar{x},\bar{y}) \; \wedge \; \eta_\varphi(\bar{z},\bar{x}) \; \right] \] equivalent to $\varphi(\bar{z})$ which would be existential if $\psi(\bar{y})$  were. This formula $\psi(\bar{y})$ is the same formula for all $\varphi$, so we have essentially reduced to finding an existential formula equivalent to a single formula.
		
		We now return to considering the recursively saturated model $\mc{B}$. Recall that there is $\bar{b} \in \mc{B}$ such that, for the formula $\psi(\bar{x})$ in particular, 
		\[ \mc{B} \models \forall \bar{x} \; \left[ \; \psi(\bar{x}) \; \longleftrightarrow \; \bigdoublevee_{i = 1}^{\infty} \theta^{\mc{B}}_{\psi,i}(\bar{x},\bar{b}) \; \right].\]
		As argued previously, since $\mc{B}$ is recursively saturated there is $n$ such that
		\[ \mc{B} \models \forall \bar{x} \; \left[ \; \psi(\bar{x}) \; \longleftrightarrow \; \bigvee_{i = 1}^{n} \theta^{\mc{B}}_{\psi,i}(\bar{x},\bar{b}) \; \right].\]
		Thus
		\[ \mc{B} \models \exists \bar{y} \; \forall \bar{x} \; \left[ \; \psi(\bar{x}) \; \longleftrightarrow \; \bigvee_{i = 1}^{n} \theta^{\mc{B}}_{\psi,i}(\bar{x},\bar{y}) \; \right].\]
		Let $\rho(\bar{x},\bar{y}) :=\bigvee_{i = 1}^{n} \theta^{\mc{B}}_{\psi,i}(\bar{x},\bar{y})$. Then, since $T$ is complete,
		\begin{equation}\label{existence}
			T \models \exists \bar{y} \; \forall \bar{x} \; \left[ \; \psi(\bar{x}) \; \longleftrightarrow \; \rho(\bar{x},\bar{y}) \; \right].
		\end{equation}
		We claim that the $\mc{L} \cup \{\bar{c}\}$-theory
		\[ T^* := T \cup \left\{ \; \forall \bar{x} \; \left[ \; \psi(\bar{x}) \; \longleftrightarrow \; \rho(\bar{x},\bar{c}) \; \right] \; \right\} \]
		is model complete. We will verify this by showing that for each formula $\varphi$ we have
		\begin{equation}\label{toprove}
			T^* \models \forall \bar{z} \; \left[ \; \varphi(\bar{z}) \; \longleftrightarrow \; \exists \bar{x} \; \exists \bar{y} \; \left(\;\rho(\bar{y},\bar{c}) \wedge \nu_\varphi(\bar{x},\bar{y}) \wedge \eta_\varphi(\bar{z},\bar{x}) \;\right)\;\right].
		\end{equation}
		where $\exists \bar{x} \; \exists \bar{y} \; \left(\;\rho(\bar{y},\bar{c}) \wedge \nu_\varphi(\bar{x},\bar{y}) \wedge \eta_\varphi(\bar{z},\bar{x}) \;\right)$ is an existential $\mc{L} \cup \{\bar{c}\}$ formula. Let $(\mc{A},\bar{a}) \models T^*$, that is, $\mc{A} \models T$ and
		\begin{equation}\label{Afact}
			\mc{A} \models \forall \bar{x} \; \left[ \; \psi(\bar{x}) \;\longleftrightarrow \; \rho(\bar{x},\bar{a}) \;\right].
		\end{equation}
		Then for each formula $\varphi$ we have, by \eqref{Tfactonenew} and \eqref{Afact},
		\begin{equation}\label{Afact2}
			\mc{A} \models \exists \bar{x} \;\exists \bar{y} \; \left[ \; \rho(\bar{y},\bar{a}) \wedge \nu_\varphi(\bar{x},\bar{y}) \;\right]
		\end{equation}
		and, by \eqref{Tfacttwonew} and \eqref{Afact},
		\begin{equation}\label{Afact3}
			\mc{A} \models \forall \bar{x} \; \forall \bar{y} \; \left[ \; \left( \; \rho(\bar{y},\bar{a}) \wedge \nu_\varphi(\bar{x},\bar{y}) \; \right) \; \longrightarrow \; \forall \bar{z} \; \left( \; \varphi(\bar{z}) \; \longleftrightarrow \; \eta_\varphi(\bar{z},\bar{x}) \; \right) \; \right].
		\end{equation}
		Combining \eqref{Afact2} and \eqref{Afact3} we get
		\[ \mc{A} \models \forall \bar{z} \; \left[ \; \varphi(\bar{z}) \; \longleftrightarrow \; \exists \bar{x} \; \exists \bar{y} \; \left( \; \rho(\bar{y},\bar{a}) \wedge \nu_\varphi(\bar{x},\bar{y}) \wedge \eta_\varphi(\bar{z},\bar{x}) \; \right) \; \right].\]
		This completes the proof of \eqref{toprove}, and so we have shown that $T^*$ is model complete. Together with \eqref{existence} we have completed the proof of the theorem.
	\end{proof}
	
	\section{Marker extensions}\label{sec:three}
	
	Let $\cL$ be a countable relational language.
	Let $T$ be an $\cL$-theory and let $\alpha_1(\bar{x}),\alpha_2(\bar{x}),\ldots$ be a distinguished set of relations from $\cL$ with $\alpha_i$ of arity $n_i$.
	
	We will define a new language $\mc{L}^+$ and an $\mc{L}^+$-theory $T^+$ which we call the \textit{Marker extension of $T$ making $\alpha_1,\alpha_2,\ldots$ into $\Sigma_1$ relations}. This new theory will essentially be the same as the original theory $T$ except that the $\alpha_i(\barx)$ will be understood as definable $\Sigma_1$ relations. Thus given some $\cA \models T$, we can find some 
	$\cA^+ \models T^+$ which is essentially the same, except we have now turned the $\alpha_i$ into $\Sigma_1$-definable relations. Marker extensions, first introduced by Marker in \cite{Marker}, are now a standard tool in computable structure theory, and there are many different ways of defining them with slightly different properties. While Marker used the construction at the level of a theory---which is what we will do---since then they have often been used on a single structure. We will give a full development in order to prove Lemma \ref{lem:qe} which as far as we know is new.
	
	Let our new language be 
	\[\cL^+ = (\cL - \set{\alpha_i(\barx)}_{i \in \omega}) \cup \set{f_{i, j}(y)}_{i \in \omega, 1 \leq j \leq n_i} \cup \set{S_1(x), S_2(x)}.\]
	Here, $S_1$ and $S_2$ are unary relations, and the $f_{i,j}$ are unary functions. We will write $f_i(y) = \bar{x}$ for the function $f_i(y) = (f_{i,1}(y),\ldots,f_{i,n_i}(y))$. At the level of formulas, $f_i(y) = \barx$ stands for $\bigwedge_{j = 1}^{n_i} f_{i, j}(y) = x_j$. Similarly we write $S_1(\barx)$ as shorthand for $\bigwedge_{i = 1}^{\abs{\barx}} S_1(x_i)$.
	
	We first explain, given an $\mc{L}$-structure $\cA$, how to define its Marker extension $\cA^+$ which will be an $\mc{L}^+$-structure. $\cA^+$ will have two sorts. $S_1$ will identify the first sort, the elements from the original structure, and $S_2$ will 
	identify the second sort, extra elements we add to our structure to certify whether a relation $\alpha_i$ is true or false. The domain of $\cA^+$ will consist of the domain of $\cA$, all satisfying the relation $S_1$, together with some number of new elements satisfying the relation $S_2$, and no other elements. (Usually we will have infinitely many new elements, but not always.) For each relation in $\mc{L} \cap \mc{L}^+$, we keep the relation the same on the domain of $\cA$, and false on any other tuple. Then, for each $\alpha_i$ and $\bar{a} \in S_1$ such that $\cA \models \alpha_i(\bar{a})$, we introduce an element $b$ of $S_2$ satisfying $f_i(b) = \bar{a}$. We say that $b$ witnesses that $\alpha_i(\bar{a})$ is true. Each element of $S_2$ will witness one and only one such relation. To fully define $\cA^+$, we set $f_{i,j}(x) = x$ for $x$ from either sort if $x$ is not witnessing a relation $\alpha_i$. We have, in $\cA^+$, definable relations $\alpha_i$ on the first sort which make the first sort isomorphic to $\cA$. Thus $\cA$ can be recovered from $\cA^+$. We call $\cA^+$ the Marker extension of $\cA$ making $\alpha_1,\alpha_2,\ldots$ into $\Sigma_1$ relations.
	
	Given the theory $T$, we now define the theory $T^+$. Given $\cA \models T$, we should have $\cA^+ \models T^+$. We begin with axioms for the sorts.
	\begin{enumerate}
		\item $T^+ \models \forall x \; [S_1(x) \vee S_2(x)]$.
		\item $T^+ \models \forall x \; [S_1(x) \longleftrightarrow \neg S_2(x)]$.
	\end{enumerate}
	For relation symbols $R$ from $\mc{L}$ which remain in $\cL^+$, they can only be satisfied by tuples from the first sort.
	\begin{enumerate}[resume]
		\item $T^+ \models \forall \barx \; [ R(\barx) \longrightarrow S_1(\barx)]$.
	\end{enumerate}
	We now write down the axioms that ensure that the $f_{i,j}$ appropriately define relations on the first sort:
	\begin{enumerate}[resume]
		\item for all $i$ and $j = 1,\ldots,n_i$,
		\[ T^+ \models \forall x \; [S_1(x) \longrightarrow f_{i, j}(x) = x],\]
		\item for all $i$,
		\[ T^+ \models \forall x \forall \bary \; [ \;[S_2(x) \wedge f_i(x) = \bary] \longrightarrow [S_1(\bary) \vee \bar{y} = (x,\ldots,x)] \;],\]
		\item for all $i$,
		\[ T^+ \models \forall \barx \forall y \forall z \; [ \; [S_1(\barx) \wedge S_2(y) \wedge S_2(z) \wedge f_i(y) = \barx \wedge f_i(z) = \barx] \longrightarrow y = z \;],\]
		\item for all $i \neq j$,
		\[ T^+ \models \forall \barx \forall  y \; [ \; [S_1(\barx) \wedge S_2(y) \wedge f_i(y) = \barx] \longrightarrow f_j(y) = (y,\ldots,y) \; ].\]
	\end{enumerate}
	We now get $\Sigma_1$-definable relations on tuples from the first sort corresponding to the $\alpha_i$. Abusing notation, we often also write $\alpha_i$ for these definable relations, but when we need to distinguish them we write $\tilde{\alpha}_i$.
	\[\tilde{\alpha}_i(\barx) := \exists y \; [S_2(y) \wedge f_i(y) = \barx]\]
	When this holds, we say that $y$ is the witness to $\alpha_i(\bar{x})$. Axioms (6) ensure that each witness is unique, while axioms (7) ensure that each element  cannot be a witness to two different relations $\alpha_i$ and $\alpha_j$. Since the $f_i$ are (tuples of) functions, each element cannot be a witness to two different relations for the same $\alpha_i$ on different tuples. We inductively define a transformation of an $\mc{L}$-formula $\varphi$ into an $\mc{L}^+$-formula $\varphi^*$ by replacing each instance of $\alpha_i$ with the definable relation $\tilde{\alpha}_i$, and with all variables of $\varphi$ interpreted as belonging to $S_1$: Each quantifier $\exists \barx\; \varphi(\barx)$ becomes $\exists \barx\; [S_1(\barx) \wedge  \varphi^*(\barx)]$, and each quantifier $\forall \barx\; \varphi(\barx)$ becomes $\forall \barx\; [S_1(\barx) \longrightarrow  \varphi^*(\barx)]$. If $\varphi(\bar{x})$ has free variables $\bar{x}$, in $\varphi^*(\bar{x})$ we also assert that $\bar{x} \in S_1$. For the final axioms of $T^+$, we take the axioms of $T$, translated as just described.
	\begin{enumerate}[resume]
		\item $T^+ \models \varphi^*$ for each $\varphi \in T$.
	\end{enumerate} 
	This completes the definition of $T^+$.
	
	We begin with some observations. If $T$ is consistent, say with a model $\cA$, then $\cA^+$ is a model of $T^+$. Thus:
	
	\begin{lemma}\label{lem:consistent}
		If $T$ is consistent, then $T^+$ is consistent.
	\end{lemma}
	
	\noindent On the other hand, given a model $\mc{M}$ of $T^+$, we can define an associated model $\cA$ whose domain is the first sort of $\mc{M}$, and interprets the $\alpha_i$ using the definable relations $\tilde{\alpha}_i$. Axioms (8) ensure that $\cA \models T$. It is not necessarily true that $\mc{M}$ is $\mc{A}^+$, but $\mc{A}^+$ will be a substructure of $\mc{M}$, with $\mc{M} - \mc{A}^+$ consisting only of elements of  the second sort $S_2$ which are not witnesses to any relations. For $n \in \omega \cup \{ \omega\}$, we write $\mc{A}^+_n$ for the structure which is $\mc{A}^+$ together with $n$ additional elements in the second sort which are not witnesses to any relations. The $\mc{A}^+_n$ for $n \geq 1$ may not be models of $T^+$ in all cases, but often by compactness they will be. One such instance is if $T$ satisfies a technical assumption we call ($*$): $T \models \exists \bar{x} \; \alpha_i(\bar{x})$ for infinitely many $i$. In this case the countable models of $T^+$ are exactly the models $\mc{A}^+_n$ for $\mc{A} \models T$ and $n \in \omega \cup \{ \omega\}$.
	
	Next, we prove the following quantifier-elimination result which says that any $\mc{L}^+$ formula about a model of $T^+$ can be expressed by a Boolean combination of quantifier-free $\mc{L}^+$-formulas and formulas (possibly involving quantifiers) in the original language $\mc{L}$ which have been translated to $\cL^+$. For this we will need to assume that $T$ satisfies ($*$).\footnote{We could make slightly different definitions in the case that there are only finitely many $\alpha_1,\ldots,\alpha_k$. In general, we could always add new relations $\beta(\bar{x})$ to the list $\alpha_i$ and assume that $T \models \forall \bar{x} \; \beta(\bar{x})$. This is only a trivial modification to $T$ but makes $T$ satisfy assumption ($*$).}
	
	\begin{lemma}\label{lem:qe}
		Suppose that $T$ has property ($*$). Then $T^+$ has quantifier elimination
		in the language  
		\[\cL^* = \cL^+ \cup \set{\vphi^*: \text{$\vphi$ an $\cL$-formula}}.\]
	\end{lemma}
	
	Sometimes one always allows at least one free variable in a quantifier elimination result. Here we do not: every sentence is equivalent to a quantifier-free sentence. This yields the following important lemma which we state and prove before returning to the proof of the proposition.
	
	\begin{lemma}\label{lem:complete}
		If $T$ is complete and has property ($*$), then $T^+$ is complete.
	\end{lemma}
	\begin{proof}
		Given an $\mc{L}^+$-sentence $\theta$, we can write $\theta$ as a Boolean combination of atomic and negated atomic formulas from
		\[\cL^+ \cup \set{\vphi^*: \text{$\vphi$ an $\cL$-formula}}\]
		with no free variables. The only atomic formulas with no free variables are the $\varphi^*$ for $\varphi$ an $\mc{L}$-sentence. But since $T$ decides each such sentence $\varphi$, $T^+$ decides each such $\varphi^*$.
	\end{proof}
	
	We now return to prove Lemma \ref{lem:qe}.
	
	\begin{proof}[Proof of Lemma \ref{lem:qe}]
		Using the standard criterion for proving quantifier elimination, it suffices to prove the following. Let $\mc{A}$ be an $\mc{L}^*$-structure and $\mc{M},\mc{N} \models T^+$. Suppose that $\mc{A} \subseteq \mc{M}$ and $\mc{A} \subseteq \mc{N}$ as $\mc{L}^*$-structures. Let $\varphi(\bar{x},y)$ be a quantifier-free $\mc{L}^*$-formula. Let $\bar{a} \in \mc{A}$ and $u \in \mc{M}$ and suppose that $\mc{M} \models \varphi(\bar{a},u)$. We must show that there is $v \in \mc{N}$ such that $\mc{N} \models \varphi(\bar{a},v)$. We may assume that $u \notin \mc{A}$, or we would already be done.
		
		We begin by making some simplifications. After writing $\vphi$ in disjunctive normal form, one of the disjuncts must be true in $\mc{M}$ and so we can replace $\varphi$ by this disjunct. Thus we assume that $\vphi$
		is a conjunction of atomic and negated atomic formulas. We may also assume that $\varphi$ has at most one conjunct of the form $\psi^*$, as a conjunction $\psi_1^* \wedge \psi_2^*$ is equivalent to $(\psi_1 \wedge \psi_2)^*$. This includes any of the relations $R$ from $\mc{L} - \{\alpha_1,\alpha_2,\ldots\}$ as for such relations $R$ and $R^*$ are equivalent. Then we can write
		\[\vphi(\bar{a}, u) \equiv \psi^*(\bar{s}(\bara,u)) \wedge \cdots\]
		where $\bar{s}$ is a tuple of terms with $\bar{s}(\bar{a},u)$ in $S_1$ and $\cdots$ is a conjunction of equalities, inequalities, and relations $S_1$ and $S_2$ of terms in $\bara, u$. Note that all of the functions in the language are unary, and so all terms are unary, and in particular, no term can involve both part of $\bar{a}$ and $u$. We can also rewrite $\bar{a}$ as $\bar{a},\bar{b}$ where $\bar{a}$ is in the first sort and $\bar{b}$ is in the second sort. We can also expand $\bar{a}$ to include all of the images of $\bar{b}$ under any terms (other than those that map $b_i$ to itself), as each element of $S_2$ has only finitely many images under terms, and each term maps an element either to itself or to $S_1$. After this, we may rewrite our formula as
		\[\vphi(\bar{a},\bar{b}, u) \equiv \psi^*(\bara,s_1(u),\ldots,s_n(u)) \wedge \cdots\]
		where $s_1,\ldots,s_n$ are terms with $s_i(u) \in S_1$ and the other conjuncts are as before.
		
		Now we make some simplifications involving the other conjuncts. Any atomic formula only involving $\bara$ is passed to $\cN$ trivially, so we can ignore these. Since $S_1$ is equivalent to $\neg S_2$, and vice versa, we may assume that they appear only positively. Since the only function symbols in the language are unary, each term has only a single variable. Any term applied to one of the $a_i$ must be $a_i$. Also, for any term $t$ and any $i$ either $t(b_i) = b_i$ or there is $j$ such that $t(b_i) = a_j$. Thus remaining atomic formulas involving $u$ are:
		\begin{multicols}{2}
			\begin{itemize}
				\item $a_i = t(u)$ or $b_i = t(u)$
				\item $a_i \neq t(u)$ or $b_i \neq t(u)$
				\item $t_1(u) = t_2(u)$
				\columnbreak
				\item $t_1(u) \neq t_2(u)$
				\item $t(u) \in S_1$,
				\item $t(u) \in S_2$.
			\end{itemize}
		\end{multicols}
		\noindent We also make several remarks about terms. Given $\mc{M} \models T^+$, $x \in \mc{M}$, and any term $t$:
		\begin{enumerate}
			\item Either $\mc{M} \models t(x) \in S_1$ or $\mc{M} \models t(x) = x$.
			\item If $\mc{M} \models x \in S_1$ then $t(x) = x$.
		\end{enumerate}
		We will use these implicitly.
		
		We now split into cases according to whether $u \in S_1$ or $u \in S_2$, and in the latter case into subcases depending on whether $u$ is a witness to some 
		$\alpha_k$ or not. In each case we find $v \in \mc{N}$ which fits into the same case (e.g., $\mc{M} \models u \in S_1$ if and only if $\mc{N} \models v \in S_1$).
		
		\begin{description}
			\item[Case 1: $u \in S_1$.]
			Since $u \in S_1$, $t(u) = u$ for any term $t$. This simplifies all of the atomic and negated atomic formulas in the conjunction. Any equalities and inequalities between $u$ and $\bar{a}$ can be moved into $\psi$, and we cannot have any equalities between $u$ and $\bar{b}$. Also, $u \in S_1$ and $u \notin S_2$. As we will choose $v \in S_1$, all of these remaining atomic and negated atomic formulas will hold. Thus, as long as we pick $v$ in $\mc{N}$ with $v \in S_1$ and such that $\mc{N} \models \psi^*(\bar{a},v)$ then we are done.
			
			We have
			\[\cM \models \psi^*(\bar{a}, u) \Longrightarrow \cM \models \exists y \; \psi^*(\bara, y) \Longleftrightarrow \cM \models (\exists y\; \psi(\barx, y))^*(\bara)\]
			Since the latter formula is quantifier-free, we have that
			\[\cM \models (\exists y \; \psi(\barx, y))^*(\bara) \Longrightarrow \cN \models (\exists y \;\psi(\barx, y))^*(\bara) \Longrightarrow \cN \models \exists y \; \psi^*(\bara, y)\]
			Thus there is some $v \in \mc{N}$ such that $\cN \models \psi^*(\bar{a}, v)$. This implies that $v \in S_1$.
			
			\item[Case 2: $u \in S_2$, $u$ is a witness to $\alpha_k(\barc)$ for $\barc \in \mc{M}$.]
			
			By saying that $u$ is a witness to $\alpha_k(\barc)$ in $\mc{M}$ we mean that $f_k(u) = \bar{c}$. We begin by splitting $\barc$ into part that is in $\bar{a}$ and part that is not. After rearranging, we may replace $\bar{c}$ by $(\bar{a}',\bar{c})$ where $\bar{a}' = (a_{i_1},\ldots,a_{i_p})$ is a subtuple of $\bar{a}$ and $\bar{c} = (c_1,\ldots,c_q)$ has no entries in common with $\bar{a}$. Also, the only terms $s(u)$ are (up to equivalence, when applied to $u$) the identity and $s = f_{k,m}$ for each $m$. Rearranging the $\{ f_{k,1},\ldots,f_{k,n_k} \}$ we may assume that $f_{k,m}(u) = a_{i_m}$ for $1 \leq m \leq p$ and $f_{k,p+m}(u) = c_{m}$ for $1 \leq m \leq q$. Moreover, for any $v$ in $S_2$, as long as $f_{k,m}(v) \in S_1$ for some $m$, then the terms $s(v)$ up to equivalence are the same. So we may rewrite our formula as
			\[\vphi(\bara, \barb, u) \equiv \psi^*(\bara,\bar{c}) \wedge \bigwedge_{1 \leq m \leq p} f_{k,m}(u) = a_{i_m} \wedge \bigwedge_{1 \leq m \leq q} f_{k,p + m}(u) = c_{m} \wedge \cdots.\]
			Thus we have
			\[\cM \models \tilde{\alpha}_k(\bar{a}',\bar{c}) \wedge \psi^*(\bar{a}, \bar{c}) \wedge \Delta(\bar{a},\bar{c}) \]
			where $\Delta$ is the formula that expresses the equalities and inequalities between all elements of $\bar{a}$ and $\bar{c}$, including that $a_i \neq c_j$ for all $i, j$.
			Thus
			\[\cM \models \exists \bar{y} \; [ \tilde{\alpha}_k(\bar{a}',\bar{y}) \wedge \psi^*(\bar{a}, \bar{y}) \wedge \Delta(\bar{a},\bar{y}) ].\]
			We can rewrite this as
			\[\cM \models ( \exists \bary \; [\alpha_k(\bar{x}',\bar{y}) \wedge \psi(\bar{x}, \bar{y}) \wedge \Delta(\bar{x},\bar{y})] )^*(\bar{a})\]
			where $\bar{x}' = (x_{i_1},\ldots,x_{i_p})$. Since this is a  quantifier-free fact about $\bar{a} \in \mc{A}$, it is also true in $\mc{N}$,
			\[\cN \models ( \exists \bary \; [\alpha_k(\bar{x}',\bar{y}) \wedge \psi(\bar{x}, \bar{y}) \wedge \Delta(\bar{x},\bar{y})] )^*(\bar{a}).\]
			Thus
			\[ \cN \models \exists \bary \; [\alpha_k(\bar{a}',\bar{y}) \wedge \psi^*(\bar{a}, \bar{y}) \wedge \Delta(\bar{a},\bar{y}) ].\]
			Letting $\bar{d} \in \mc{N}$ be the witness to this existential quantifier over $\bar{y}$, we have
			\[ \cN \models \tilde{\alpha}_k(\bar{a}',\bar{d}) \wedge \psi^*(\bar{a}, \bar{d}) \wedge \Delta(\bar{a},\bar{d}).\]
			Note that the $d_i$ are distinct from the $a_j$ and $c_i = c_j$ if and only if $d_i = d_j$.
			
			We also have $\mc{N} \models \tilde{\alpha}_k(\bar{a}',\bar{d})$. Let $v \in \mc{N}$ be the witness to this, i.e., the unique element such that $f_k(v) = (\bar{a}',\bar{d})$. So we have
			\[ \mc{N} \models \bigwedge_{1 \leq m \leq p} f_{k,m}(v) = a_{i_m} \wedge \bigwedge_{1 \leq m \leq q} f_{k,p + m}(v) = d_{m}.\]    
			We also have that $v \notin \bar{b}$, argued as follows. In that case, we would have $v \in \mc{A}$ and so $v \in \mc{M}$. By our assumption that $\bar{a}$ contains the non-trivial images of $\bar{b}$ under all terms, we would have that $\bar{c}$ and $\bar{d}$ are the empty tuple. But $\alpha_k(\bar{a}')$ can have only one witness in $\mc{M}$, and both $u$ and $v$ would be witnesses. This contradicts our assumption that $u \notin \mc{A}$.
			
			Now we must consider all of the other conjuncts in our formula and use the fact that they were true for $u$ to show that they are true for $v$. It will be useful to note that for any term $t$, $t(u) = u$ if and only if $t(v) = v$, and otherwise there is some $m$ such that $t(u) = f_{k,m}(u)$ and $t(v) = f_{k,m}(v)$. Thus $t(u) = a_i$ if and only if $t(v) = a_i$, and $t(u) = c_i$ if and only if $t(v) = d_i$.
			\begin{itemize}
				\item Given $a_i = t(u)$ for some term $t$, we must have that $t(u) = f_{k,m}(u)$ for some $m$ as otherwise $t(u)$ would not be in $S_1$. Thus $a_i = a_{i_m} = f_{k,m}(v) = t(v)$.
				
				\item Given $b_i = t(u)$, since $u$ and $b_i$ are in $S_2$, we must have $b_i = u$. But this contradicts our assumption that $u \notin \mc{A}$.
				
				\item Given $a_i \neq t(u)$, one of the following must be true:
				\begin{itemize}
					\item $t(u) = u$ and $t(v) = v \neq a_i$.
					\item $t(u) = f_{k,p+m}(u) = c_m$ for some $1 \leq m \leq q$. Then $t(v) = f_{k,p+m}(v) = d_m$ and $d_m \neq a_i$.
					\item $t(u) = f_{k,m}(u) = a_{i_m}$ for some $1 \leq m \leq p$, and $a_{i_m} \neq a_i$. Then $t(v) = f_{k,m}(v) = a_{i_m} \neq a_i$.
				\end{itemize}
				
				\item Given $b_i \neq t(u)$ for some term, we have that either $t(v) \in S_1$ or $t(v) = v$. In the former case, we automatically have that $b_i \neq t(v)$. In the latter case, we have two subcases. If $\bar{d}$ is a non-empty tuple, then since $\bar{d}$ does not share elements with $\bara$ and is generated by $v$, $v \neq b_i$, since $b_i$ can only generate itself and elements of $\bara$. If $\bar{d}$ is empty, then so is $\barc$. If $\barc$ is empty, then $u$ is the unique witness in $\mc{M}$ to $\alpha_k(\bar{a}')$, and so $b_i$ is not a witness to $\alpha_k(\bar{a}')$. Since $v$ is the unique witness in $\mc{N}$ to $\alpha_k(\bar{a}')$, $b_i \neq v$.
				
				\item Given $t_1(u) = t_2(u)$ (or $t_1(u) \neq t_2(u)$), since terms behave essentially the same on $v$ as they do on $u$, (including that, e.g., $c_i = c_j$ if and only if $d_i = d_j$) we also have $t_1(v) = t_2(v)$ (or $t_1(v) \neq t_2(v)$).
				\item Given $t(u) \in S_1$ we must have $t(u) = f_{k,m}(u)$ and so $t(v) = f_{k,m}(v)$ is also in $S_1$.
				\item Given $t(u) \in S_2$ we must have $t(u) = u$, and so $t(v) = v$ is also in $S_2$.
			\end{itemize}
			
			Thus we have shown that $\mc{N} \models \varphi(\bar{a},\bar{b},v)$.
			
			\item[Case 3: $u \in S_2$, $u$ is not a witness to any $\alpha_k$.]
			For every term $t$, $t(u) = u$ is in $S_2$. Thus we may rewrite
			\[\vphi(\bara, \barb, u) \equiv \psi^*(\bara) \wedge \cdots\]
			and since $\psi^*(\bar{a})$ depends only on elements in $\mc{A}$ and hence is automatically true in $\mc{N}$, we may remove it and assume that $\vphi(\bara, \barb, u)$ is a conjunction of atomic and negated atomic formulas of the types described above.
			
			Since, by assumption, $u$ is not part of the tuple $\bar{b}$, and is not a witness, we can remove the equalities of the form $a_i = t(u)$. Since in $\mc{M}$ for every term $t$ we have $t(u) = u$, no inequalities $t_1(u) \neq t_2(u)$ can be true in $\mc{M}$, and so we can discard those as a possibility. We can also replace any equalities $t_1(u) = t_2(u)$ with two equalities of the form $t(u) = u$. Finally, since $t(u) = u$ is in $S_2$ in $\mc{M}$, we can discard atomic formulas $t(u) \in S_1$. We are left with the following:
			\begin{multicols}{3}
				\begin{itemize}
					\item $a_i \neq t(u)$ or $b_i \neq t(u)$
					\columnbreak
					\item $t(u) = u$
					\columnbreak
					\item $t(u) \in S_2$.
				\end{itemize}
			\end{multicols}
			Let $K$ be the maximum of the $k$ such that some $f_{k,i}$ shows up in one of the terms $t$ in our formula, and also so that any $b_i \in \bar{b}$ is a witness to some $\alpha_k$. It will suffice to choose $v$ to be a witness to some $\alpha_m$ for some $m > K$, and that such an element exists by assumption ($*$)---the original theory $T$ was assumed to contain elements that witness $\alpha_m$ for arbitrarily large $m$. In this case, for all terms $t$ appearing in our formula, we have $t(v) = v$. Also, $v \in S_2$, and $v$ does not appear among the $\bar{b}$, and so we are done.
		\end{description}
		
		This completes the three cases. We have proved quantifier elimination down to $\cL^*$.
	\end{proof}
	
	The intention when taking a Marker extension is that given $\mc{A}$ an $\mc{L}$-structure we should be able to treat the relations $\alpha_i$ as being existential. If an existential formula is true in a substructure, it is true in a superstructure; but it may be false in a substructure and yet true in a superstructure. We define a new notion of substructure, which we call a weak substructure, to take this into account.
	
	\begin{definition}
		Let $\cA$ and $\cB$ be $\cL$-structures. We say $\cA$ is a weak substructure of $\cB$, denoted as $\cA \subseteq_w \cB$ if 
		\begin{enumerate}
			\item The domain of $\mc{A}$ is contained in the domain of $\mc{B}$;
			\item if $\cA \models \alpha_i(\bara)$ for $\bar{a} \in \mc{A}$, then $\cB \models \alpha_i(\bara)$; and
			\item for any other relation $R \in \cL - \set{\alpha_1, \alpha_2, \ldots}$ and $\bar{a} \in \mc{A}$, $\mc{A} \models R(\bar{a})$ if and only if $\mc{B} \models R(\bar{a})$.
		\end{enumerate}
	\end{definition}
	
	By taking Marker extensions, we can realize weak substructures as true substructures.
	
	\begin{lemma}\label{lem:weak-substructure}
		If $\cA \subseteq_w \cB$ then $\cA^+ \subseteq \cB^+$.
	\end{lemma}
	\begin{proof}
		Supposing that $\mc{A}$ is a weak substructure of $\mc{B}$, we must define an embedding of $\mc{A}^+$ into $\mc{B}^+$. The embedding is defined as follows. On the first sort $S_1$, we just have the embedding $\mc{A} \to \mc{B}$ as these make up the first sorts of $\mc{A}^+$ and $\mc{B}^+$ respectively. On the second sort $S_2$, any element $c$ of $\mc{A}^+$ is a witness that $\alpha_i$ holds of some tuple $\bar{a} \in \mc{A}$. Since $\mc{A}$ is a weak substructure of $\mc{B}$, $\alpha_i$ also holds of $\bar{a}$ in $\mc{B}$, and so there is a witness $d$ to this in $\mc{B}^+$. We map $c$ to $d$. Each $c$ witnesses $\alpha_i(\bar{a})$ for exactly one pair $i$ and $\bar{a}$, and there is only one witness $d$ to the same pair in $\mc{B}^+$. Thus this map $\mc{A}^+ \to \mc{B}^+$ is uniquely defined and injective. It is also not hard to see that it is an $\mc{L}^+$-embedding.
	\end{proof}
	
	The previous lemma also, of course, applies when $\mc{A} \subseteq \mc{B}$ is a true substructure.

	\begin{lemma}\label{lem:marker-subtheory}
		Let $S \models T$ as $\cL$-theories. Then $S^+ \models T^+$ as $\cL^+$-theories.
	\end{lemma}
	\begin{proof}
		Axioms ($1$) and ($2$) are common to the Marker extension of any theory. Axioms ($3$)-($7$) are satisfied by both $S^+$ and $T^+$ since those 
		axioms only depend on $\cL$ and set $\set{\alpha_1, \alpha_2, \ldots}$ of relations to be made existential. Consider an instance of Axiom $(8)$, 
		say $S^+ \models \vphi^*$ where $S \models \varphi$. Since $S \models T$, $T \models \varphi$ and so $T^+ \models \varphi^*$ as an instance of Axiom ($8$).
	\end{proof}
	
	\section{Incomplete theories}\label{sec:four}
	
	In this section we will prove the anti-classification result for incomplete theories.
	
	\secondtheorem*
	
	\begin{proof}
		Given a tree $\treeT \subseteq \mathbb{N}^{<\mathbb{N}}$ we may assume, without changing whether or not $\treeT$ is well-founded, that each non-leaf node $\sigma \in \treeT$ has all of its possible children $\sigma i$ in $\treeT$. Uniformly computably in $\treeT$ we will define a language $\mc{L} = \mc{L}_\treeT$ and an (incomplete) theory $T = T_\treeT$ such that $\treeT$ is well-founded if and only if $T_{\treeT}$ is relatively decidable. Thus the set
		\[ \{ (\mathcal{L},T) : \text{the $\mathcal{L}$-theory $T$ is relatively decidable}\}\]
		will be $\boldsymbol{\Pi}^1_1$-Wadge-complete. One can see from the construction below that the language always consists of the same number of symbols of the same arities, and so we could make the construction with a fixed language. Indeed, following \cite{BazhenovHT} we conjecture that one could take the language $\mathcal{L}$ to be finite.
		
		Fix a tree $\treeT$ with the property that each non-leaf node has all of its children. Let
		\[ \cL = \set{U_\sigma(x): \sigma \in \treeT} \cup \set{A_i(x): i \in \omega} \cup \set{V_{i,\sigma}(\bary, x): i \in \omega, \sigma \in \treeT}.\]
		For each $k$ let
		\[ \mc{L}_k = \set{U_\sigma: \sigma \in \set{0, 1, 2, \ldots, k - 1}^{\leq k} \cap \treeT} \cup \{A_i : i < k\} \cup \{V_{i, \sigma} : i < k, \sigma \in \set{0, 1, \ldots, k - 1}^{\leq k} \cap \treeT\}.\]
		The $U_\sigma$ and $A_i$ are unary relations. The arities of the $V_{i, \sigma}$ must be defined inductively. Enumerate the 
		universal $\mathcal{L}$-formulas as $P_i(\barx)$  with $P_i(\barx)$ an $\mathcal{L}_{i}$-formula. Thus $P_i(\barx)$ does not involve any of the symbols $V_{i,\sigma}$. Then we may construct $\mc{L}$ so that the arity of $\bary$ in $V_{i, \sigma}(\bary,x)$ is the same as in $P_i(\bary,x)$. Let $T$ be the theory with the following axioms:
		\begin{enumerate}
			\item $\exists x \; U_{\varnothing}(x)$,
			\item for $\sigma \in \treeT$ a non-leaf node,
			\[ \left[ \; U_\sigma(x) \; \wedge \; \bigwedge_{j = 0}^i \forall y \; \neg U_{\sigma j}(y) \; \right] \;  \longrightarrow \; \forall \bary \; \left[ \; P_i(\bary) \longleftrightarrow V_{i, \sigma}(\bary, x) \; \right],\]
			\item for $\tau \in \treeT$ a leaf node,
			\[  U_\tau(x) \; \longrightarrow \; \forall \bary \; \left[\; P_i(\bary) \longleftrightarrow V_{i,\tau}(\bary, x)\;\right].\]
		\end{enumerate}
		We begin by showing that $T$ is satisfiable.
		
		\begin{lemma}
			$T$ is satisfiable and has an infinite model.
		\end{lemma}
		\begin{proof}
			We will define a model $\mc{A} \models T$ on an infinite domain. Set $U_\varnothing(x)$ for all $x \in \mc{A}$ and $\neg U_\sigma(x)$ for all $x \in \mc{A}$ and all $\sigma \in \treeT$, $\sigma \neq \varnothing$. This ensures that axiom (1) is true. We set $A_i(x)$ to be true for all $i$ and all $x \in \mc{A}$.
			
			We will define the $V_{i, \sigma}$ such that for all $x \in \mc{A}$,
			\[ \mc{A} \models \forall \bar{y} \; \left[\; P_i(\bary) \longleftrightarrow V_{i, \sigma}(\bar{y},x) \; \right].\]
			This will ensure that the axioms of types (2) and (3) are all true. Because the symbols $V$ are included in the formulas $P$, we must define the $V$'s recursively. We begin with $P_0$ which uses only $U_\varnothing$ and the symbol for equality, so that we have already determined whether $\mc{A} \models P_0(\bary)$ for each $\bary$. We can thus set, for each $\bary$ and $x$, 
			$\mc{A} \models V_{0, \sigma}(\bar{y},x)$ if and only if $\mc{A} \models P_0(\bary)$. Now supposing that we have defined the $V_{i, \sigma}$ for $i < k$, we define the $V_{k, \sigma}$. We have already at this point determined whether $\mc{A} \models P_k(\bary)$ for each $\bary$, and so can again set $\mc{A} \models V_{k, \sigma}(\bar{y},x)$ if and only if $\mc{A} \models P_k(\bary)$.
		\end{proof}
		
		Now let $T^+$ be the Marker extension of $T$ where we make each $A_i$ and $U_\sigma$ universal (i.e., we make each of their negations existential). Since $T$ is consistent, so is $T^+$. We will now argue that $T^+$ is relatively decidable if and only if $\treeT$ is well-founded.
		
		\begin{lemma}
			Suppose that $\treeT$ is well-founded. Then $T^+$ is relatively decidable.
		\end{lemma}
		\begin{proof}
			Let $\mc{A}^+ \models T^+$ and let $\mc{A}$ be the corresponding model of $T$. We have that $\mc{A} \models \exists x \; U_\varnothing(x)$. Since the tree $\treeT$ is well-founded, there must either be some $\sigma \in \treeT$ a non-leaf node such that $\mc{A} \models \exists x \; U_\sigma(x)$ and for each $k$, $\mc{A} \models \forall x \; \neg U_{\sigma k}(x)$, or some $\tau \in \treeT$ a leaf node such that $\mc{A} \models \exists x \; U_\tau(x)$.
			
			In the first case, let $c$ be such that $\mc{A} \models U_\sigma(c)$. Then, following axioms (2), we have for all $i$ that
			\[ \mc{A} \models \forall \bar{y} \; \left[ \; P_i(\bar{y}) \longleftrightarrow V_{i, \sigma}(\bary,c) \; \right].\]
			In the second case, let $c$ be such that $\mc{A} \models U_{\tau}(c)$. Then, following the axioms (3), we also have that for all $i$
			\[ \mc{A} \models \forall \bar{y} \; \left[ \; P_i(\bar{y}) \longleftrightarrow V_{i, \tau}(\bary,c) \; \right].\]
			From here the argument is the same in both cases. In any computations below, we fix the element $c$ and $\sigma$ (or $\tau$, but below we will write $\sigma$) non-uniformly.
			
			Each universal $\mc{L}$-formula is equivalent to a quantifier-free $\mc{L}$-formula $V_{i, \sigma}(\bar{y},c)$ with parameter $c$. We can argue inductively that this implies that each $\mc{L}$-formula $\varphi(\bary)$ is equivalent to an $\mc{L}$-formula $V_{i_\varphi,\sigma}(\bar{y},c)$. (With $\varphi(\bar{y})$ in normal form, work from the inside out replacing each universal formula by some $V_{i,\sigma}$ and each existential formula by some $\neg V_{i,\sigma}$.)
			Using $T^+$ (which is Turing-equivalent to $T$) for each $\varphi$ we can compute such an $i_\varphi$.
			
			Given an $\mc{L}^+$-formula $Q(\bary)$, by Lemma \ref{lem:qe} there is an equivalent quantifier-free formula $P(\bary)$ in the language $\mc{L}^+ \cup \{ \varphi^* : \text{$\varphi$ an $\mc{L}$-formula}\}$. Using $T$, we can compute such a formula. Each $\varphi^*(\bary)$ is equivalent in $\mc{A}^+$ to the $\mc{L}^+$-formula $V_{i_\varphi,\sigma}^*(\bary,c)$, which is the same as $V_{i_\varphi,\sigma}(\bary,c)$ as they were not made universal in the Marker extension. Thus $Q(\bary)$ is equivalent in $\mc{A}^+$ to a quantifier-free $\mc{L}^+$-formula with parameter $c$. Thus the atomic diagram of $\mc{A}^+$, together with $T$, can compute the elementary diagram of $\mc{A}^+$.
		\end{proof}
		
		The more difficult case is the case that $\treeT$ is ill-founded. In this case, we will show that there is a complete theory $S^+$ extending $T^+$ which is not relatively decidable. This will imply that $T^+$ is not relatively decidable. The advantage of finding a completion of $T^+$ is that we can apply our characterization for when a complete theory is relatively decidable. We will apply one direction of that characterization, which is that we will show that there is no formula $\varphi(\bar{x})$ such that 
		\begin{itemize}
			\item $S^+ \models \exists \bar{x} \; \varphi(\bar{x})$, and
			\item the $\mc{L}^+ \cup \{\bar{c}\}$-theory $S^+ \cup \{\varphi(\bar{c})\}$ is model complete.
		\end{itemize}
		The theory $S^+$ will be defined from a fixed path through $\treeT$.
		
		Instead of directly finding a completion of $T^+$, we will first find a completion of $T$ and then take its Marker extension, which is justified by Lemma \ref{lem:marker-subtheory}. Suppose that $\treeT$ has an infinite path $\pi$. Let $T_\pi$ be the theory with the following axioms.
		\begin{enumerate}
			\item for each $\sigma \nprec \pi$,
			\[\neg \exists x \; U_\sigma(x)\]
			\item for each $\sigma \prec \pi$, with $\ell$ such that $\sigma \ell \prec \pi$, and $i < \ell$,
			\[ U_\sigma(x) \longrightarrow \forall \bary \; \left[ \; P_i(\bary) \longleftrightarrow V_{i, \sigma}(\bary, x) \; \right].\]
		\end{enumerate}
		Note that $T_\pi \cup \{\exists x \; U_\sigma(x) : \sigma \prec \pi\} \models T$.\footnote{We omit $\exists x U_\sigma(x)$ from $T_\pi$ because we will be working with Fra\"{i}ss\'{e} limits and will want some associated theory to be universal. Nevertheless, as one might expect these sentences will end up satisfied by the Fra\"{i}ss\'{e} limits anyway.} Thus to find a completion of $T$ it suffices to find a completion of $T_\pi \cup \{\exists x \; U_\sigma(x) : \sigma \prec \pi\}$.
		
		We could attempt to take a model completion of $T_\pi$, which is a $\forall_2$ theory, but to do this we would have to show that the existentially closed models of $T_\pi$ are an elementary class. This is not straightforward and we are not sure if this is true. (The difficulty stems from the fact that the $P_i$ range over all universal formulas.) Instead, we will construct a completion of $T_\pi$ by Fra\"{i}ss\'{e} limits. The language $\mc{L}$ is infinite, so rather than trying to take a Fra\"{i}ss\'{e} limit of $T_\pi$ itself, we will find a completion of $T_\pi$ which is a union of Fra\"{i}ss\'{e} limits. And since $T_\pi$ is not a universal theory we cannot simply take the Fra\"{i}ss\'{e} limits of $T_\pi$ restricted to finite languages, but must instead do something more complicated.
		
		We recall some definitions from Fra\"{i}ss\'{e} theory. We work purely in the case where $\mc{L}$ is a finite relational language. Recall that the age of an $\mc{L}$-structure $\mc{A}$ is the set $\age(\mc{A})$ of all finite substructures of $\mc{A}$. Given $\mathbb{K}$ a class of (isomorphism types) of finite structures, we say that $\mathbb{K}$ is a Fra\"{i}ss\'{e} class if it has the Hereditary Property (HP), Amalgamation Property (AP), and the Joint Embedding Property (JEP). Then there is a unique countable homogeneous structure, called the Fra\"{i}ss\'{e} limit of $\mathbb{K}$, whose age is $\mathbb{K}$. We write $\flim(\mathbb{K})$ for the Fra\"{i}ss\'{e} limit of the class $\mathbb{K}$. Since we assume that the language was finite and relational, $\age(\mathbb{K})$ is automatically uniformly locally finite and hence the theory of the Fra\"{i}ss\'{e} limit is $\omega$-categorical and admits quantifier elimination. Given a universal theory $T$ we have the corresponding class $\mathbb{K}(T)$ of finite models of $T$. $\mathbb{K}(T)$ automatically has the HP.
		
		Given two Fra\"{i}ss\'{e} classes $\mathbb{K}_1$ and $\mathbb{K}_2$ in languages $\mc{L}_1$ and $\mc{L}_2$  respectively, we say that $\mathbb{K}_2$ is an \textit{expansion class} of $\mathbb{K}_1$ if $\mathbb{K}_1 = \mathbb{K}_2 \res_{\mc{L}_1} := \{ \mc{A} \res_{\mc{L}_1} \;|\; \mc{A} \in \mathbb{K}_2\}$, the class of reducts of structures in $\mathbb{K}_2$. If $\mathbb{K}_2$ is an expansion class of $\mathbb{K}_1$ we say that the pair $(\mathbb{K}_1,\mathbb{K}_2)$ is \textit{reasonable} if whenever $\mc{A}_2 \in \mathbb{K}_2$, $\mc{A}_1 = \mc{A}_2 \res_{\mc{L}_1}$ is the reduct to $\mc{L}_1$, and $f \colon \mc{A}_1 \to \mc{B}_1$ is an embedding of $\mc{A}_1$ into another $\mc{B}_1 \in \mathbb{K}_1$, there is an expansion of $\mc{B}_1$ to an $\mc{L}_2$-structure $\mc{B}_2 \in \mathbb{K}_2$ such that $f$ is an embedding $\mc{A}_2 \to \mc{B}_2$. The Fra\"{i}ss\'{e} limit of $\mathbb{K}_1$ is the reduct to $\mc{L}_1$ of the Fra\"{i}ss\'{e} limit of $\mathbb{K}_2$ if and only if $(\mathbb{K}_1,\mathbb{K}_2)$ is \textit{reasonable}.\footnote{This definition seems to mostly appear in work on big Ramsey degrees where this fact is used without proof. However one can find a proof in Proposition 5.3 of \cite{Zucker}.}

		Recall that
		\[ \mc{L}_k = \set{U_\sigma: \sigma \in \set{0, 1, 2, \ldots, k - 1}^{\leq k} \cap \treeT} \cup \{A_i : i < k\} \cup \{V_{i, \sigma} : i < k, \sigma \in \set{0, 1, \ldots, k - 1}^{\leq k} \cap \treeT\}.\]
		Note that $\mc{L}_0 = \{U_\varnothing\}$. Let $T_k = T_\pi \res \mc{L}_k$ be the restriction of $T_\pi$ to the axioms of $T_\pi$ with symbols from $\mc{L}_k$, namely with the following axioms from $T_\pi$:
		\begin{enumerate}
			\item[($1'$)] for each $\sigma \nprec \pi$ and $\sigma \in \{0,1,2,\ldots,k-1\}^{\leq k}$,
			\[\neg \exists x \; U_\sigma(x),\]
			\item[($2'$)] for each $\sigma \prec \pi$, $\sigma \in \{0,1,2,\ldots,k-1\}^{\leq k}$, with $\ell$ such that $\sigma\ell \prec \pi$, and $i < \min(\ell,k)$,
			\[ U_\sigma(x) \longrightarrow  \forall \bary \; \left[ \; P_i(\bary) \longleftrightarrow V_{i, \sigma}(\bary, x) \; \right].\]
		\end{enumerate}
		Any axiom of $T_\pi$ which does not appear as an axiom of $T_k$ includes some symbol from $\mc{L} - \mc{L}_k$. (We do not worry at this point about whether $T_\pi$ has some consequences in $\mc{L}_k$ which are proven by axioms not in $\mc{L}_k$. This will be resolved by Claim \ref{claim:three}.)
		
		We will recursively define universal $\mc{L}_k$-theories $T_k^*$ for $k \in \mathbb{N}$. For each $k$, $\mathbb{K}(T_k^*)$ will be a Fra\"{i}ss\'{e} class with $(\mathbb{K}(T_k^*),\mathbb{K}(T_{k+1}^*))$ being reasonable. Let $S_k \supseteq T_k^*$ be the $\mc{L}_k$-theory of the Fra\"{i}ss\'{e} limit of $\mathbb{K}(T_k^*)$. This is a complete, consistent extension of $T_k^*$ which is $\omega$-categorical and admits quantifier elimination. Since the Fra\"{i}ss\'{e} limit of $\mathbb{K}(T_k^*)$ is a reduct to $\mc{L}_k$ of the Fra\"{i}ss\'{e} limit of $\mathbb{K}(T_{k+1}^*)$, we will have that $S_k = S_{k+1} \res_{\mc{L}_k}$.
		
		We will begin with $T_0^* = T_0$ and will ensure that $T_{k+1}^* \cup S_k \models T_{k+1}$. Since $S_{k + 1} \models T_{k+1}^*$, as it is the theory of the Fra\"{i}ss\'{e} limit of $T_{k+1}^*$, and because $S_{k + 1} \models S_k$ by reasonableness, we will have $S_{k + 1} \models T_{k+1}$. Letting $S = \bigcup_{k \in \omega} S_k$, and after verifying that $S \models \exists x \; U_\sigma(x)$ for every $\sigma \prec \pi$, we will get a complete $\mc{L}$-theory with $S \models T_\pi \cup \{  \exists x \; U_\sigma(x) : \sigma \prec \pi \}$ and hence $S \models T$.
		
		Beginning with $T_0^* = T_0$, this has no axioms as $\mc{L}_0$ has only the symbol $U_\varnothing$. This is of course a Fra\"{i}ss\'{e} class and has as its Fra\"{i}ss\'{e} limit infinitely many elements satisfying, and infinitely many elements not satisfying, $U_\varnothing$. Now given the $\mathcal{L}_{k-1}$-theory $S_{k-1}$ we will define $T_{k}^*$. Since $S_{k-1}$ admits quantifier elimination, there is a quantifier-free $\mc{L}_{k-1}$-formula $Q_{k-1}(\bar{y})$ such that $S_{k-1} \models P_{k-1}(\bar{y}) \longleftrightarrow Q_{k-1}(\bar{y})$. We will have already defined, for each $i < k-1$, an $\mc{L}_i$-formula $Q_i(\bar{y})$ such that $S_i \models P_i(\bar{y}) \longleftrightarrow Q_i(\bar{y})$. Since $S_{k-1} \models S_i$, this equivalence is also true modulo $S_{k-1}$. We are now ready to define $T_{k}^*$, which has the following axioms:
		\begin{enumerate}
			\item[($1^*$)] for each $\sigma \nprec \pi$, $\sigma \in \{0,1,2,\ldots,k-1\}^{\leq k}$,
			\[ \neg \exists x \; U_\sigma(x) \]
			\item[($2^*$)] for each $\sigma \prec \pi$, $\sigma \in \{0,1,2,\ldots,k-1\}^{\leq k}$, with $\ell$ such that $\sigma \ell \prec \pi$, and $i < \min(\ell,k)$
			\[U_\sigma(x) \longrightarrow \forall \bary \;[Q_i(\bary) \longleftrightarrow V_{i, \sigma}(\bary, x)].\]
		\end{enumerate}
		($1^*$) is exactly the same as ($1'$), and ($2^*$) is the same as ($2'$) except with $Q_i(\bary)$ replacing $P_i(\bary)$. So $T_{k}^*$ includes, for each axiom of $T_{k}$, an axiom which is equivalent modulo $S_{k-1}$. Thus $T_{k}^* \cup S_{k-1} \models T_{k}$ and $T_{k} \cup S_{k-1} \models T_{k}^*$. We also have that $T_{k}^* \models T_{k-1}^*$ as each axiom of $T_{k-1}^*$ is an axiom of $T_{k}^*$.
		
		The point of defining $T_{k}^*$ is that it is a universal theory and hence $\mathbb{K}(T_{k}^*)$ has the HP. The next two lemmas complete the verification that $\mathbb{K}(T_{k}^*)$ is a Fra\"{i}ss\'{e} class.
		
		\begin{claim}
			$\mathbb{K}(T_{k}^*)$ has the JEP.
		\end{claim}
		\begin{proof}
			Let $\cA, \cB \models T_k^*$. We define a sort of free amalgam $\cC$ of $\cA$ and $\cB$. The domain of $\cC$ is the disjoint union $C = A \sqcup B$ of the domains of $\mc{A}$ and $\mc{B}$.
			For any tuple contained entirely in $\mc{A}$ or $\mc{B}$, let $\cC$ inherit the relations from $\cA$ and $\cB$. Since the $U_\sigma$ and $A_i$ are unary relations, this defines them completely on $\cC$. It remains to define the relations $V_{i, \sigma}$ for tuples $(\bary, x)$ that are not fully contained in $\cA$ or $\cB$. On these tuples we would like to set these relations so that $\cC \models V_{i, \sigma}(\bary,x) \longleftrightarrow Q_i(\bary)$, but we must do this recursively so that when we are defining $V_{i, \sigma}$ we have already determined $Q_i(\bary)$. First, for $V_{0, \sigma} \in \cL_k$ and $(\bary, x)$ not fully contained in $\cA$ or $\cB$, set $\cC \models V_{0, \sigma}(\bary,x)$ if and only if $\cC \models Q_0(\bary)$. $Q_0$ is an $\mc{L}_0$-formula, and the reduct of $\cC$ to $\mc{L}_0 = \{U_\varnothing\}$ is already defined. Once all $V_{i, \sigma} \in \mc{L}_k$ for $i < n$ have been determined, we can set $\cC \models V_{n, \sigma}(\bary,x)$ if and only if $\cC \models Q_n(\bary)$ which is well-defined as $Q_n$ is an $\mc{L}_n$-formula and the reduct of $\cC$ to $\mc{L}_n$ has already been defined. We can see that for all $(\bary, x)$ not fully included in $\cA$ or $\cB$ we have $\cC \models Q_i(\bary) \longleftrightarrow V_{i, \sigma}(\bary, x)$.
			
			We must now check that $\cC \models T_k^*$. The axioms ($1^*$) are immediate, since the $U_\sigma$ are unary and each element of $\cC$ is in either $\cA$ or $\cB$ both of which satisfy ($1^*$). For ($2^*$), suppose that $\cC \models U_\sigma(x)$ and $x \in \mc{A}$. For $\bary \in \cA$, since $\cA \models T_k^*$ we have $\cC \models Q_i(\bary) \longleftrightarrow V_{i, \sigma}(\bary, x)$, and for tuples $\bary$ not completely 
			contained in $\mc{A}$ by construction we have $\cC \models Q_i(\bary) \longleftrightarrow V_{i, \sigma}(\bary, x)$. The case where $x \in \mc{B}$ is identical. So we have verified that $\cC \models T_k^*$.
		\end{proof}
		
		\begin{claim}
			$\mathbb{K}(T_{k}^*)$ has the AP.
		\end{claim}
		\begin{proof}
			Given $\cA,\cB,\cC \models T_{k}^*$ with $\cA \subseteq \cB$ and $\cA \subseteq \cC$, we must form the amalgam $\cD$ of $\cB$ and $\cC$ over $\cA$. This will be a free amalgam of $\cB$ and $\cC$ with domain $D = B \cup C$. It is clear this free amalgam makes the diagram commute, so all that is left to do is assign relations.
			
			As for the JEP, let $\cD$ inherit the relations from $\cB, \cC$ for tuples solely contained in $\cB, \cC$. For tuples entirely contained in their intersection $\cA$, we know that $\cB, \cC$ agree on the relations. It only remains to assign tuples that are not entirely contained in either $\cB$ or $\cC$. As before, we inductively assign $V_{i, \sigma}$ on such tuples $(\bary, x)$ so that $\cD \models Q_i(\bary) \longleftrightarrow V_{i, \sigma}(\bary, x)$. The argument that $\cD \models T_k^*$ is again the same as for the JEP, with ($1^*$) being relatively immediate and ($2^*$) following from the fact that for new tuples $(\bary, x)$ we have $\cD \models Q_i(\bary) \longleftrightarrow V_{i, \sigma}(\bary, x)$.
		\end{proof}
		
		Since $T_{k}^* \models T_{k-1}^*$, for any $\mc{A} \in \mathbb{K}(T_{k}^*)$ the reduct of $\mc{A}$ to $\mc{L}_{k-1}$ is in $\mathbb{K}(T_{k-1}^*)$. So $ \mathbb{K}(T_{k}^*)$ is an expansion class of $\mathbb{K}(T_{k-1}^*)$. We now verify that it is reasonable.
		
		\begin{claim}\label{claim:three}
			$(\mathbb{K}(T_{k-1}^*),\mathbb{K}(T_{k}^*))$ is reasonable.
		\end{claim}
		\begin{proof}
			Let $\cA, \cB \models T_{k-1}^*$ be such that $\cA \subseteq \cB$. 
			Let $\cA^* \models T_{k}^*$ be an expansion of $\cA$ to an $\mc{L}_{k}$-structure. We must construct $\mathcal{B}^* \models T_{k}^*$ an expansion of $\mc{B}$ such that $\mathcal{A}^* \subseteq \mathcal{B}^*$. Note that
			\begin{align*}\cL_{k} - \cL_{k-1} = \set{U_\sigma&: \sigma \in \treeT  \cap \set{0, 1, \ldots, k-1}^{\leq k} \setminus \set{0, 1, \ldots, k - 2}^{\leq k-1 }} \cup \set{A_{k-1}} \\
				&\cup \set{V_{i, \sigma}: i < k-1, \sigma \in \treeT \cap \set{0, 1, \ldots, k-1}^{\leq k} \setminus \set{0, 1, \ldots, k - 2}^{\leq k-1}} \\
				& \cup  \set{V_{k-1, \sigma}: \sigma \in \treeT \cap \set{0, 1, \ldots, k-1}^{\leq k}}.\end{align*}
			To define the relations on $\mc{B}^*$, let all relations from $\mc{A}^*$ carry over. For tuples not fully contained in $\mc{A}^*$, let $A_{k-1}$
			be false and all new $U_\sigma$ be false. It remains to set the $V_{i, \sigma}$ that are introduced in $\mc{L}_{k}$ for tuples that are not exclusively in $\mc{A}^*$. Since the quantifier-free formulas $Q_i$, $i < k$, only use symbols from $\cL_{i} \subseteq \cL_{k-1}$, it is already determined whether $\cB \models Q_i(\bary)$. For tuples $(\bary,x)$ not entirely contained in $\cA^*$ we can simply set $\cB^* \models V_{i, \sigma}(\bary,x)$ if and only if $\cB \models Q_i(\bary)$.
			
			It remains to verify that $\cB^* \models T_{k}^*$. We already know that $\cB \models T_{k-1}^*$, so we just need to check the new axioms. New axioms of the form ($1^*$) are all of the form $\neg \exists x \; U_\sigma(x)$ with $U_\sigma \in \mc{L}_{k} - \mc{L}_{k-1}$. For $x \in \cA^*$ we have $\neg U_\sigma(x)$ as $\cA^* \models T_{k}^*$, and for $x \notin \cA^*$ recall that we set $\neg U_\sigma(x)$.
			
			Now we consider axioms of the form ($2^*$). Such an axiom is of the form
			\[ U_\sigma(x) \longrightarrow \forall \bary \;[ \; Q_i(\bary) \longleftrightarrow V_{i, \sigma}(\bary, x) \; ]\]
			where $\sigma \prec \pi$, $\ell$ is such that $\sigma \ell \prec \pi$, and $i < \min(\ell,k)$. 
			It is equivalent to check that for each $x \in \cB^*$ and $i < \min(\ell,k)$ that
			\[\cB^* \models U_\sigma(x) \longrightarrow \forall \bary \; [ \; Q_i(\bary) \longleftrightarrow V_{i, \sigma}(\bary, x) \;]\]
			For such an axiom to be new, one of $U_\sigma$ or $V_{i, \sigma}$ are in $\cL_{k} - \cL_{k-1}$.
			If $U_\sigma \in \cL_{k} - \cL_{k-1}$, then $\sigma \in \treeT \cap \set{0, 1, \ldots, k-1}^{\leq k} - \set{0, 1, \ldots, k-2}^{\leq k-1}$, 
			so $V_{i, \sigma} \notin \cL_{k-1}$. Thus for any such new axiom we have $V_{i, \sigma} \in \cL_{k} - \cL_{k-1}$ is a new symbol. Suppose $x \in \cA^*$ and $U_\sigma(x)$. Then for $\bary$ entirely contained in $\cA^*$ we have
			$\cA^* \models Q_i(\bary) \longleftrightarrow V_{i, \sigma}(\bary, x)$ since $\cA^* \models T_{k}^*$.
			For $\bary$ not completely contained in $\mc{A}^*$, by construction we have that $\cB^* \models Q_i(\bary) \longleftrightarrow V_{i, \sigma}(\bary, x)$. Thus $\cB^* \models \forall \bary \; [Q_i(\bary) \longleftrightarrow V_{i, \sigma}(\bary, x)]$. If $x \notin \cA^*$, then whether or not we have $\cB^* \models U_\sigma(x)$, by construction $\cB^* \models \forall \bary \; [Q_i(\bary) \longleftrightarrow V_{i, \sigma}(\bary, x) ]$. Thus $\cB^* \models T_{k}^*$, so $(\mathbb{K}(T_{k-1}^*),\mathbb{K}(T_{k}^*))$ is reasonable.
		\end{proof}
		
		Let $S_{k}$ be the complete theory of the Fra\"{i}ss\'{e} limit of $\mathbb{K}(T_{k}^*)$, which is $\omega$-categorical and admits quantifier elimination. Since $(\mathbb{K}(T_{k-1}^*),\mathbb{K}(T_{k}^*))$ is reasonable, $S_{k} \res_{\mc{L}_{k-1}} = S_{k-1}$. Let $S = \bigcup_k S_k$. This is a complete $\mc{L}$-theory. Since each $S_k$ admits quantifier elimination, so does $S$.
		
		\begin{claim}
			For each $\sigma \prec \pi$, $S \models \exists x \; U_\sigma(x)$.
		\end{claim}
		\begin{proof}
			Fix $\sigma \prec \pi$. It suffices to show that for some $k$, $S_k \models \exists x \; U_\sigma(x)$. Since $S_k$ is the Fra\"{i}ss\'{e} limit of $T_k^*$, it suffices to show that there is a model of $\mc{A} \models T_k^*$ with $\mc{A} \models \exists x \; U_\sigma(x)$. Let $k$ be large enough that $U_\sigma \in \mc{L}_k$. We can take $\mc{A}$ to have a single element $a$ satisfying $U_\sigma$ and no other $U_\tau \in \mc{L}_k$. This element $a$ will also satisfy no $A_i \in \mc{L}_k$, and we can inductively define the $V_{i,\tau} \in \mc{L}_k$ such that $\mc{A} \models Q_i(a,\ldots,a) \longleftrightarrow V_{i,\tau}(a,\ldots,a,a)$. Thus, checking axioms of types ($1^*$) and ($2^*$), we see that $\mc{A} \models T_k$.
		\end{proof}
		
		For each $k$, $S \models T_{k}^* \cup S_{k - 1} \models T_{k}$ and so $S \models T_\pi$. Hence $S$ is a completion of $T_\pi \cup \{\exists x \;U_\sigma(x) : \sigma \prec \pi\}$ and hence, as discussed earlier, $S$ is a completion of $T$.
		
		Now let $S^+$ be the Marker extension of $S$ where we make each $A_i$ and $U_\sigma$ universal. Since $S$ is consistent, so is $S^+$ (Lemma \ref{lem:consistent}). Also $S$ has property ($*$) as each $\sigma$ of length $1$ is in $\treeT$ (as $\varnothing$ has a child in $\treeT$ and hence all of its possible children are in $\treeT$) but $S \models \forall x \; \neg U_{\sigma}(x)$ for all $\sigma \nprec \pi$. Recall that it is the $\neg U_\sigma$ that are made existential in the Marker extension. $S^+$ is complete and by Lemma \ref{lem:marker-subtheory} it is a completion of $T^+$. We will now argue that $S^+$ is not relatively decidable.
		
		\begin{claim}
			$S^+$ is not relatively decidable.
		\end{claim}
		\begin{proof}
			It suffices to show that for each formula $\varphi(\bar{x})$ such that $S^+ \models \exists \bar{x} \; \varphi(\bar{x})$, the $\mc{L}^+ \cup \{\bar{c}\}$-theory $S^+ \cup \{\varphi(\bar{c})\}$ is not model complete. Fix some such formula $\varphi(\bar{x})$. Let $k$ be such that $\varphi(\bar{x})$ is an $\mc{L}^+_k$-formula. By Lemma \ref{lem:qe}, $\varphi(\bar{x})$ is equivalent to a quantifier-free formula in the language $\mc{L}_k^+ \cup \{ \psi^* : \text{$\psi$ an $\mc{L}_k$-formula}\}$. Since $S_k$ admits quantifier elimination, $\varphi(\bar{x})$ is in fact equivalent to a quantifier-free formula in the language $\mc{L}_k^+ \cup \{ \psi^* : \text{$\psi$ a quantifier-free $\mc{L}_k$-formula}\}$. Choose $j > k$ sufficiently large so that for any $\sigma \prec \pi$ with $U_\sigma \in \cL_k$, $j$ is greater than any entry of $\sigma$ and also greater than the next entry of $\pi$ following all such $\sigma$.
			
			We will show that $S^+ \cup \{\varphi(\bar{c})\}$ is not model complete by showing that $A_j$ is not equivalent to an existential formula.  Given $\cA^+ \models S^+$ and $\barc \in \cA^+$ such that $\cA^+ \models \varphi(\barc)$ and with some $d \in \cA^+$ such that $\cA^+ \models A_j(d)$, we will construct a model $\cC^+ \supseteq \cA^+$ with $\cC^+ \models S^+ \cup \{\varphi(\bar{c})\}$ and  such that $\cC^+ \models \neg A_j(d)$.
			
			We will now argue that we may assume that $\bar{c}$ is entirely contained within the first sort of $\mc{A}^+$. If not, we may also split $\bar{c}$ into a pair of tuples $\bar{c}_1$ in the first sort $S_1$ and $\bar{c}_2$ in the second sort $S_2$, so that $\varphi$ can be written as $\varphi(\bar{x}_1,\bar{x}_2)$. Consider all of the atomic and negated atomic terms, in the language $\mc{L}_k^+ \cup \{ \psi^* : \text{$\psi$ a quantifier-free $\mc{L}_k$-formula}\}$, which appear in $\varphi(\bar{x}_1,\bar{x}_2)$. Since $\mc{A}^+ \subseteq \mc{C}^+$, the truth value of any atomic or negated atomic formula from $\mc{L}_k^+$ is maintained from $\mc{A}^+$ to $\mc{C}^+$. It is only the formulas $\psi^*$ which may change their truth value. Any such formula can only be true of elements in the first sort, and so if applied to a tuple including any element of $\bar{c}_2$ must be false in both $\mc{A}^+$ and $\mc{C}^+$. Thus, essentially, we only need to be concerned with $\bar{c}_1$, as $\bar{c}_2$ poses no obstacle.
			
			Let $\cA$ be the model of $S$ associated with $\cA^+$.  Let $\cB$ be the same as $\cA$ except that $\cB \models \neg U_\sigma(x)$ for all $x$ of the appropriate sort where $U_\sigma \notin \cL_k$, and also $\cB \models \neg A_j(d)$. Then $\cA$ is a weak substructure of $\cB$,  $\cA \subseteq_w \cB$. If $\cB^+$ is the Marker extension of $\cB$, then by Lemma \ref{lem:weak-substructure}, $\cA^+ \subseteq \cB^+$. Since $j > k$ the reducts of $\cA$ and $\cB$ to $\mc{L}_k$ are the same, and as $\varphi$ is an $\cL_k^+$-formula, we have $\cB^+ \models \varphi(\barc)$. Now we argue that $\cB$ satisfies the universal part of $S$.
			
			\begin{subclaim}
				$\cB \models S_\forall$.
			\end{subclaim}
			\begin{proof}
				Since $S = \bigcup_m S_m$, any universal consequence of $S$ is also a consequence of one of the $\mc{L}_m$-theories $S_m$. The theories $S_m$ are the Fra\"{i}ss\'{e} theories of the Fra\"{i}ss\'{e} classes $\mathbb{K}(T_{m}^*)$, and so it suffices to show that $\cB \models T_m^*$ for every $m$. 
				
				Using the fact that the reduct of $\cA$ to $\mc{L}_m$ is a model of $T_m^*$, we argue that $\cB$ is also a model of $T_m^*$. We satisfy the axioms of type $(1^*)$ trivially since none of the elements of $\cA$ satisfied $U_\sigma(x)$ for $\sigma \not\prec \pi$, and in $\cB$ we did not add any new elements or change any relations from $\neg U_\sigma$ to $U_\sigma$. Recalling the form of axioms of type $(2^*)$, for each $\sigma \prec \pi$ and $\ell$ such that $\sigma \ell \prec \pi$, and $i < \ell$, we will verify that
				\[\cB \models \forall x \; [U_\sigma(x) \longrightarrow \forall \bary[Q_i(\bary) \lra V_{i, \sigma}(\bary, x)]].\]
				Note that these are more axioms than appear in $T_m^*$, as we are only assuming that $i < \ell$ and not $i < \min(\ell,m)$, but as these axioms are all true anyway there is no reason to assume $i < m$. We will only consider those $U_\sigma \in \cL_k$, since all other $U_\sigma$ were set to false in $\mc{B}$. Such $\sigma$ are of length at most $k$ and satisfy $\sigma \in \set{0, 1, \ldots, k - 1}^{\leq k}$, so $\sigma \ell$ is at most length $k + 1$.
				We do not have a bound on $\ell$, though we do know that $j > \ell$.
				
				Fix $\sigma$ and $x \in \cB$ and assume $\cB \models U_\sigma(x)$. Consider some $Q_i(\bary)$ for $i < \ell$. It is a quantifier-free formula in $\cL_i$. Our goal is to make sure that $\cB \models Q_i(\bary) \lra V_{i, \sigma}(\bary, x)$ for each $i < \ell$. Since we did not change the truth value of $V_{i, \sigma}$, it suffices to show that we did not change the truth value of $Q_i(\bary)$. Since $Q_i$ is quantifier-free in $\cL_i$, it suffices to show that we did not change the truth value of any atomic formula that $Q_i$ is composed of. If $\ell \leq k$, then we know $Q_i$'s truth value is preserved since the reducts of $\cA$ and $\cB$ to $\cL_k$ are equal. Thus assume $\ell \geq k + 1$. First note that all $A_i$ for $i < \ell$ remain the same as in $\cA$ (we only changed $A_j$, with $j > \ell$) as do all $V_{i, \sigma}$ for any $i, \sigma$. Thus by ensuring that any $U_{\tau}$ in $Q_i$ has the same truth value as it does in $\cA$, we know that $Q_i$ will have the same truth value as $V_{i, \sigma}$. We only consider $\tau \prec \pi$ since all others are false in both $\mc{A}$ and $\mc{B}$. We claim that if $U_\tau$ appears in $Q_i$ for $i < \ell$ then $U_{\tau} \in \cL_k$. For the sake of contradiction, assume that $U_{\tau} \in \cL - \cL_k$. Thus $\tau \in \treeT \setminus \set{0, 1, \ldots, k - 1}^{\leq k}$. Either $\tau$ has length at least $k+1$ or there is $n \geq k$ such that $n$ appears as an entry in $\tau$. In the second case, note that $\sigma$ does not have $n$ as an entry and so $\sigma \prec \tau$.
				
				\begin{case}
					$\tau$ has length $\geq k+1$. Since $\tau \prec \pi$ and is of length $\geq k + 1$, we know $\tau$ must be $\sigma \ell$ or an extension of it. However, recall that $U_{\tau} \in \mc{L}_{i} \subseteq \cL_{\ell - 1}$, which only contains $U_\rho$ for $\rho \in \set{0, 1, \ldots, \ell - 2}^{\leq \ell - 1}$, so $\tau$ cannot possibly contain $\ell$. This yields a contradiction.
				\end{case}
				\begin{case}
					There is $n \geq k$ such that $n$ appears in $\tau$. Recall that $\ell \geq k + 1$. Since $\sigma \in \{0,1,\ldots,k-1\}^{\leq k}$, $\sigma$ cannot contain $n$ as any entry, and so $\tau$ cannot be an initial segment of $\sigma$. Since $\sigma$ and $\tau$ are compatible, we have $\sigma \prec \tau \prec \pi$. Thus $\sigma \ell \preceq \tau$. Since $U_{\tau}$ is contained in $Q_i$, a quantifier-free formula in $\cL_i$, it must be that $\tau \in \set{0, 1, \ldots, \ell - 2}^{\leq \ell - 1}$, since $i < \ell$. This yields a contradiction.
				\end{case}
				Thus we have shown that $U_{\tau} \in \cL_k$. We conclude that $Q_i$ has the same truth value in $\mc{B}$ as in $\cA$ and so axioms of the form $(2^*)$ are fulfilled. Since $\cB \models T_m^*$ for each $m$, $\cB \models S_\forall$, completing the proof of the subclaim.
			\end{proof}
			
			Since $\cB \models S_\forall$, standard model-theoretic results imply that there is $\cC \supseteq \cB$ with $\cC \models S$. Let $\cC^+$ be the Marker extension of $\cC$, so that $\cC^+ \models S^+$, with $\cA^+ \subseteq \cB^+ \subseteq \cC^+$. We have $\cC^+ \models \neg A_j(d)$. It remains to argue that $\cC^+ \models \varphi(\bar{c})$. Recall that in models of $S^+$, $\varphi$ is equivalent to a quantifier-free formula in the language $\mc{L}_k^+ \cup \{ \psi^* : \text{$\psi$ a quantifier-free $\mc{L}_k$-formula}\}$. So it suffices to show that $\cA^+$ is a substructure of $\cC^+$ in this language, as then since $\cA^+ \models \varphi(\barc)$, we can conclude that $\cC^+ \models \varphi(\barc)$. Since $\cA^+$ is an $\cL^+$-substructure of $\cC^+$, it remains to check that if a relation $\psi^*$ holds of some tuple in $\cA^+$ then it holds in $\cC^+$ for $\psi$ a quantifier-free $\cL_k$-formula. This follows from the fact that, after taking their reducts to $\cL_k$, $\cA$ is a substructure of $\cB$ and $\cB$ is a substructure of $\cC$. This completes the argument of the claim.
		\end{proof}
		
		This completes the proof of the theorem: If $\treeT$ is well-founded, then the theory $T^+$ constructed is relatively decidable, and if $\treeT$ is ill-founded then the theory $S^+ \supseteq T^+$ is not relatively decidable and hence $T^+$ is not relatively decidable.
	\end{proof}
	
	\bibliography{References}
	\bibliographystyle{alpha} 
	
\end{document}